\documentclass[10pt,a4paper]{amsart}
\usepackage[english]{babel}
\usepackage[T1]{fontenc}
\usepackage{indentfirst}
\usepackage{fancyhdr}
\usepackage{amsfonts}
\usepackage{amsmath}
\usepackage{amsthm, mathrsfs}
\usepackage{amssymb, enumitem}
\usepackage{hyperref}
\usepackage{theoremref}
\usepackage{bbm}
\usepackage[normalem]{ulem}
\newtheorem{thm}{Theorem}[section]
\newtheorem{lem}[thm]{Lemma}
\newtheorem{cor}[thm]{Corollary}
\newtheorem{prop}[thm]{Proposition}

\newtheorem{probl}{}

\theoremstyle{definition}
\newtheorem{rem}[thm]{Remark}
\newtheorem{dfn}[thm]{Definition}

\numberwithin{equation}{section}
\newcommand{\be}{\begin{equation}}
\newcommand{\ee}{\end{equation}}

\def\NN{\mathbb{N}}

\def\RR{\mathbb{R}}

\def\B{{\mathcal B}}

\def\P{{\mathcal P}}

\def\C{{\mathcal C}}

\def\dd{{\mathrm{d}}}
\def\tr{{\mathrm{tr}}}
\newcommand{\ind}{\text{\Large $\mathbbm{1}$}}

\newcommand{\Pos}{\mathrm{Pos}}
\newcommand{\supp}{\mathrm{supp}}
\renewcommand{\sp}{\mathfrak{sp}}

\usepackage{tikz}
\usepackage{tikz-cd}

\makeatletter
\@namedef{subjclassname@2020}{%
  \textup{2020} Mathematics Subject Classification}
\makeatother

\title[Moment problem for algebras generated by a nuclear space]{\bf Moment problem for\\ algebras generated by a nuclear space}
\author[Infusino, Kuhlmann, Kuna, Michalski]{Maria Infusino, Salma Kuhlmann, Tobias Kuna, Patrick Michalski}
\address[M. Infusino]{\newline \indent Dipartimento di Matematica e Informatica, Universit\`{a} degli Studi di Cagliari,\newline \indent
Via Ospedale 72,
09124 Cagliari, Italy.}
\email{maria.infusino@unica.it}
\address[S. Kuhlmann, P. Michalski]{\newline \indent Fachbereich Mathematik und Statistik, Universit\"at Konstanz,\newline \indent
Universit\"atstrasse 10,
78457 Konstanz, Germany.}
\email{salma.kuhlmann@uni-konstanz.de}
\email{patrick.michalski@googlemail.com}
\address[T. Kuna]{\newline \indent Dipartimento di Ingegneria, Scienze dell' Informazione e Matematica,\newline \indent Universit\`{a} degli Studi dell'Aquila, Via Vetoio, 67100, L'Aquila,
Italy}
\email{tobias.kuna@univaq.it}

\date{\today}

\begin{document}
\begin{abstract}
We establish a criterion for the existence of a representing Radon measure for linear functionals defined on a unital commutative real algebra $A$, which we assume to be generated by a vector space $V$ endowed with a Hilbertian seminorm $q$.  Such a general criterion provides representing measures with support contained in the space of characters of $A$ whose restrictions to $V$ are $q-$continuous. This allows us in turn to prove existence results for the case when $V$ is endowed with a nuclear topology. In particular, we apply our findings to the symmetric tensor algebra of a nuclear space.
\end{abstract}
\maketitle

\section*{Introduction}

\noindent
Given a unital commutative (not necessarily finitely generated) $\RR$--algebra~$A$ and a linear subspace $V$ of $A$, we say that \emph{$A$ is generated by $V$} if there exists a set of generators $G$ of $A$ such that $V$ is the linear span of $G$, i.e. $V=span(G)$. Equivalently, $A$ is generated by $V$ if $V$ contains a set of generators of $A$. This article deals with the moment problem for $A$ generated by a vector space $V$ which is endowed with a topology $\tau_V$ compatible with the addition and the scalar multiplication, namely $(V, \tau_V)$ is a topological vector space. Moreover, we always assume that the \emph{character space} of~$A$, i.e., the set~$X(A)$ of all $\RR$--algebras homomorphisms from~$A$ to~$\RR$, is non-empty and we always endow~$X(A)$ with the weakest (Hausdorff) topology~$\tau_{X(A)}$ on $X(A)$ such that for each~$a\in A$ the function~$\hat{a}\colon X(A)\to\RR$, $\alpha\mapsto\alpha(a)$ is continuous and consider on $X(A)$ the Borel $\sigma-$algebra $\B(\tau_{X(A)})$ w.r.t. $\tau_{X(A)}$. Our main question is the following.

\begin{probl}\thlabel{GenKMP}\ \\
Let $(V, \tau_V)$ a topological vector space and $A$ be an algebra generated by $V$ such that $\left\{\alpha\in X(A): \alpha\restriction_V\ \text{is} \ \tau_V-\text{continuous}\right\}\neq\emptyset$. Given a linear functional $L$ on $A$ with~$L(1)=1$, does there exist a Radon measure~$\nu$ on $X(A)$ with support contained in $\left\{\alpha\in X(A): \alpha\restriction_V\ \text{is} \ \tau_V-\text{continuous}\right\}$ such that
\begin{equation}\label{MP::eq1}
L(a)=\int_{X(A)}\hat{a}(\alpha)\dd\nu(\alpha)\qquad\text{for all }a\in A?
\end{equation}
\end{probl}

If a Radon measure~$\nu$ as in~\eqref{MP::eq1} does exist, then we call~$\nu$ a \emph{representing Radon measure for~$L$}. We recall that a \emph{Radon measure}~$\nu$ on~$X(A)$ is a non-negative measure on the Borel~$\sigma$--algebra w.r.t.\ $\tau_{X(A)}$ that is locally finite and inner regular w.r.t.\ compact subsets of $X(A)$. The \emph{support of~$\nu$}, denoted by~$\supp(\nu)$, is the smallest closed subset~$C$ of~$X(A)$ for which $\nu(X(A)\setminus C)=0$.

The main difficulty is to understand how different choices of $\tau_V$ as well as different topological properties of $L$ impact the solvability of \ref{GenKMP} and the support of the corresponding representing measures. In this article we first focus on the case when $\tau_V$ is the topology generated by a Hilbertian seminorm (i.e. a seminorm induced by a symmetric positive semidefinite bilinear form) and then consider the case when $(V, \tau_V)$ is a nuclear space, as nuclear topologies are generated by a system of Hilbertian seminorms.

An early study of \ref{GenKMP} for $(V, \tau_V)$ nuclear can be found e.g. in \cite{KM60}, \cite[Chapter 5, Section~2]{BeKo88}, \cite{BeSi71}, \cite{BY75}, \cite{H75}, \cite[Section~12.5 and 15.1]{S90}, \cite{AH08}, where $A$ is the symmetric (tensor) algebra {$S(V)$} of $V$. This is a very natural choice as $S(V)$ is isomorphic to the ring of polynomials having as variables the {coordinate} vectors with respect to a basis of $V$ and the character space $X(S(V))$ of $S(V)$ can be identified with the algebraic dual $V^*$ {of $V$}. In fact, in those works the nuclearity assumption on $V$ on $L$ allow to get the existence of representing measures with support contained in $V'$, where $V'$ is the topological dual of $V$. More recently, the role of the nuclearity assumption on $V$  was discussed in \cite[Sections 5 and 6]{Smu17}, \cite{GhInKuMa18} and {\cite[Section 3]{IK-probl}} %Question 3.7, 3.9, 3.10, 3.12,
while in \cite{ikr14} and \cite{ik20} a better localization of the support was obtained for a specific choice of the nuclear space, namely $V=\C_c^\infty(\RR^n)$, i.e. the space of infinitely differentiable functions with compact support.

{Let us describe the two main results} in this article.

{First, when $\tau_V$ is the topology generated by a Hilbertian seminorm $q$,} {we derive in \thref{MainThm}} a criterion for the existence of a representing measure with support contained in the characters of $A$ whose restrictions to $V$ are $q-$continuous (see also \thref{rem-mainThm} and \thref{MainThm-supp}). {The criterion is based on the} projective limit approach to the moment problem introduced in \cite{InKuKuMi22}, that is, {we build the representing measure for $L$ on $A$ from representing measures for $L$ restricted to finitely generated subalgebras of $A$. In fact,} we prove that a representing measure for $L$ exists if and only if for any finitely generated subalgebra $S$ of $A$ the restriction $L\restriction_S$ is represented by a Radon measure $\nu_S$ such that the family of all $\nu_S$'s is concentrated w.r.t.\! another $q-$continuous Hilbertian seminorm {$p$, which has finite trace with respect to $q$}. The concentration of a family of measures is a classical concept in measure theory and is crucial for us, because it ensures the applicability of our projective limit approach in \cite{InKuKuMi22} {by implying a Prokhorov type condition}.

{Second, in \thref{MainThm-supp_2} we show that, when $A$ itself is endowed with a Hilbertian seminorm $q$ and there exists $C>0$ such that $L(a^2)\leq Cq(a)^2$ for all $a\in A$, it is enough to check the conditions in our criterion only on a dense subalgebra of $A$ to get the existence of a
representing measure for $L$ with support contained in the $q-$continuous characters of $A$ (see also \thref{MainThm-dense-supp}).}

{These two main results are based on two Hilbertian seminorms $q$ and $p$. We investigate different choices of them in terms of the functional $L$. For example, a natural choice for $p$ is Hilbertian seminorm induced by $L$, i.e. $s_L(a):=\sqrt{L(a^2)}\,\text{ for all}\ a\in A$.} {For this choice, the concentration of the $\nu_S$'s holds automatically and so we get more concrete sufficient conditions for the existence of a representing measure for $L$ in \thref{theorem-main-s_L} and \thref{cor-dense-mainthm}}. {We then exploit in \thref{thm-um-generating-subspace} the choice of $s_L$ to demonstrate how one can give sufficient conditions only in terms of $L$ to guarantee existence of a representing measures for $L\restriction_S$ for all finite subalgebras $S$.} Those corollaries all reveal the fundamental role played by the {Hilbertian seminorm $q$.} {Thus, in the last part of Section \ref{sec:hilb-sem}, we explore the case when no Hilbertian seminorm $q$ on $V$ is pre-given. In particular,}{ in \thref{cor-appl-lemma} we give conditions under which one can construct a suitable $q$ and derive a solution for \ref{GenKMP} in this case.}

Another setting in which it is always possible to obtain a suitable $q$ is when $(V, \tau_V)$ is a nuclear space. Therefore, in Section \ref{sec:nucl}, we prove analogous results for \ref{GenKMP} when $A$ is generated by a nuclear space $(V, \tau_V)$ (see \thref{thm-nucelar-criterium::cor}, \thref{thm-um-generating-subspace::cor} and \thref{thm-um-generating-subspace::cor2}). From those corollaries, some of the results in literature mentioned above can be retrieved. 

The structure of the paper is as follows. 

In Section~\ref{sect1}, we present our general context, thereby providing definitions and notations. In particular, we review the notions of Hilbertian seminorm and nuclear space in Subsection \ref{subs1.1}. In Subsection~\ref{subs1.2}, we state and prove \thref{lem::fund-lem-Umemura} (about the support localization of a Radon probability measure on a finite dimensional space with a Hilbertian seminorm), which we need for the proof of our main theorem \thref{MainThm}. Section~\ref{sec::main-results} contains our main results, as described above. Subsection~\ref{sec:conc} is dedicated to the concept of $p$-concentration of a family of Radon measures for a given seminorm $p$, which is exploited in the subsequent Subsections \ref{sec:hilb-sem} and \ref{sec:nucl} when studying \ref{GenKMP} for $V$ endowed with the topology $\tau_V$ induced by a Hilbertian seminorm (respectively, a nuclear topology). %to prove that the Prokhorov condition (see \cite[Section 1.2]{InKuKuMi22}) is satisfied whenever $\{\nu_S: \ S\in J\}$ fulfills the concentration property. 
In Section~\ref{Sec:3} we apply our main results to the case when $A$ is the symmetric algebra $S(V)$ of a nuclear space $(V,\tau_V)$, see \thref{cor-SV-first} and \thref{simil-Berezansky}. In~\thref{BKS-thm}, we consider the case when some of the sufficient conditions for the existence of the representing measure for $L$ on $S(V)$ are only given on a total subset $E$ of the nuclear space $(V,\tau_V)$. Then the nuclearity allows us to obtain a Hilbertian norm $q$ on $V$ but, in order to apply our criterion \thref{MainThm-supp_2} to the dense sub-algebra $S(\mathrm{span}(E))$, we need a Hilbertian seminorm $\tilde{q}$ on $S(V)$, which we construct in \thref{lemma-aux}. We note that \thref{BKS-thm} is a  generalization of the classical solution to \ref{GenKMP} when $A=S(V)$ with $(V, \tau_V)$ nuclear due to Berezansky and Kondratiev. Finally, in Subsection \ref{sec: app-trace} of the Appendix \ref{sec:appendix}, we explain the relation between the notion of trace of a Hilbertian seminorm w.r.t. to another and the classical definition of trace of a positive continuous operator on a Hilbert space. We then compare in Subsection \ref{sec:app-nuclear} the definition of nuclear space used in this article (due to Yamasaki \cite{Ya85}) with that due to Grothendieck \cite{Gr55} and Mityagin \cite{Mi61}, as well as with the definitions of this concept given by Berezansky and Kondratiev in \cite[p.~14]{BeKo88} and by Schm\"udgen in \cite[p.~445]{Smu17} (this comparision is needed in Section \ref{Sec:3}). We also provide in Subsection~\ref{aux} a complete proof of the measure theoretical identity \eqref{eq-sigma-algebras}, which we exploited in the proof of \thref{MainThm-supp_2}.

\section{Preliminaries}\label{sect1}

In this section we collect some fundamental concepts, notations, and results which we will repeatedly use in the following.\par\medskip

\noindent
Throughout this article $A$ denotes a unital commutative $\RR$--algebra with non-empty character space $X(A)$.
\par\medskip

A subset $Q\subseteq A$ is a \emph{quadratic module (in $A$)} if $1\in Q, Q+Q\subseteq Q$, and $A^2Q\subseteq Q$.
The set $\sum A^2$ of all finite sums of squares of elements in $A$ is the smallest quadratic module in $A$.
The \emph{non-negativity} set of a quadratic module $Q$ is defined as
$$
K_Q:=\{\alpha\in X(A): \hat{a}(\alpha)\geq 0\text{ for all }a\in Q\}\subseteq X(A),
$$
which is closed.
Given $C\subseteq X(A)$ closed, the set
$$
\Pos(C):=\{a\in A: \hat{a}(\alpha)\geq 0\text{ for all }\alpha\in C\}
$$
is a quadratic module with $K_{\Pos(C)}=C$ (see, e.g.\ \cite[Proposition~2.1-(i)]{InKuKuMi22}).
\par\medskip

Throughout this article each linear functional $L\colon A\to\RR$ is assumed to be \emph{normalized}, that is, $L(1)=1$.

Given a a quadratic module $Q$ in $A$, we say that a linear functional $L\colon A\to\RR$ is \emph{$Q$--positive} if $L(Q)\subseteq [0,\infty)$.
In particular, each $\sum A^2$--positive linear functional $L\colon A\to\RR$ satisfies the Cauchy--Bunyakovsky--Schwarz inequality, i.e.,
\begin{equation}
L(ab)^2\leq L(a^2)L(b^2)\qquad\text{for all }a,b\in A.
\end{equation}
\par\medskip

Throughout this article we consider $A$ generated by an $\RR$--vector space $V$ endowed with a \emph{locally convex} topology, namely a topology induced by a family of seminorms. Therefore, let us recall that a function $p\colon V\to[0,\infty)$ is a \emph{seminorm} if $p(\lambda v)=\left|\lambda\right|p(v)$ and $p(v+w)\leq p(v)+p(w)$ for all $\lambda\in\RR$ and all $v,w\in V$. We denote by $B_r(p)$ the closed semi-ball of radius $r>0$ centered at the origin in $(V, p)$, i.e. $B_r(p):=\{v\in V: p(v)\leq r\}$.

A linear functional $l\colon V\to\RR$ is \emph{continuous} w.r.t.\ a seminorm $p$ on $V$ if there exists $C>0$ such that $\left| l(v)\right|\leq Cp(v)$ for all $v\in V$. We denote by $V^\prime_p$ the topological dual of $(V, p)$, i.e. the collection of all $p-$continuous linear functionals on $V$, while $V^\ast$ denotes the algebraic dual of $V$. The operator seminorm $p'$ on $V^\prime_p$ is defined as $p'(\ell):=\sup_{v\in B_1(p)}|\ell(v)|<\infty$. The weak topology on the algebraic dual (resp., topological dual) of $(V, p)$ is the weakest topology on $V^\ast$ (resp., on $V^\prime_p$ ) such that for each $v\in V$ the evaluation function $\mathrm{ev}_v\colon V^\ast\to\RR$ (resp., $V^\prime_p\to\RR$) is continuous.

We will often use the restriction map $\phi_V: X(A)\to V^\ast$ defined by $\phi_V(\alpha)=\alpha\restriction_V,$ $\forall \alpha\in X(A)$. Note that $\phi_V$ is continuous as $X(A)$ is endowed with $\tau_{X(A)}$ and $V^\ast$ with the weak topology.

We recall that the \emph{spectrum of a seminorm $p$} is defined as $$\mathfrak{sp}(p):=\{\alpha\in X(A): \alpha\text{ is }p\text{--continuous}\}.$$ More generally, for each $C>0$ we define
$$
\mathfrak{sp}_C(p):=\{\alpha\in X(A):\left|\alpha(a)\right|\leq Cp(a),\ \forall a\in A\},
$$
which is compact in $X(A)$, as it is closed and continuously embeds into the product $\prod_{a\in A}[-Cp(a),Cp(a)]$. Note that the spectrum $\mathfrak{sp}(p)=\bigcup_{n\in\NN}\mathfrak{sp}_n(p)$, which provides that $\mathfrak{sp}(p)$ is $\sigma-$compact in $X(A)$ and so Borel measurable.

\subsection{Hilbertian seminorms and nuclear spaces.}\label{subs1.1}\ 

Throughout this section $V$ will denote a real vector space.
\begin{dfn}\thlabel{dfn::Hilbertian-seminorm}
A seminorm $p$ on $V$ is called \emph{Hilbertian} if it is induced by a symmetric positive semidefinite bilinear form $\langle\cdot,\cdot\rangle$ on $V$, i.e., $p(v)=\sqrt{\langle v,v\rangle}$ for all $v\in V$.
\end{dfn}

Note that a seminorm $p$ on $V$ is Hilbertian if and only if $p$ fulfills the parallelogram law, i.e. $p(v+w)^2+p(v-w)^2=2p(v)^2+2p(w)^2$ for all $v,w\in V$, in which case the bilinear form ${\langle \cdot,\cdot\rangle_p}$ is uniquely determined by $p$ via the polarization identity: \begin{equation}\label{eq::corres-Hseminorm-bilinear}
\langle v,w\rangle=\tfrac{1}{2}\left(p(v+w)^2-p(v)^2-p(w)^2\right)\quad\text{for all }v,w\in V.
\end{equation}
For this reason, in the following we denote the positive semidefinite bilinear form inducing $p$ by $\langle \cdot,\cdot\rangle_p$.

The term ``Hilbertian seminorm'', used e.g. in~\cite{Um65} and~\cite{Ya85}, is also sometimes replaced by the term ``prehilbertian seminorm'' according to the Bourbaki's tradition~\cite[V.4, Definition 3]{BouTVS}. Both terms hints to the fact that this type of seminorms can be aways used to construct a Hilbert space (see \thref{rem::hilbertian}). \par\medskip

Let us also observe that there always exists an Hilbertian seminorm on every non-trivial vector space $V$. Indeed, if $(e_i)_{i \in I}$ is an algebraic basis of $V$ then for any $x= \sum_{i \in I} x_i e_i\in V$ and $y= \sum_{i \in I} y_i e_i\in V$ we can define
$\langle x, y \rangle := \sum_{i \in I} x_i y_i$. As only finite many summands are unequal to zero, the sum is finite and $p(x):=\sqrt{\langle x, x \rangle}$ defines a Hilbertian seminorm on $V$.%Actually there always exists a family of Hilbertian seminorms inducing a nuclear topology.

Let us now introduce the notion of trace of a Hilbertian seminorm w.r.t. to another one (see \cite[V.58, No. 9]{BouTVS}) which will be fundamental in the definition of a nuclear space used in this article. To this purpose, let us recall that
given a Hilbertian seminorm $p$ on $V$, a subset $E$ of $V$ is called:
\begin{itemize}
\item \emph{$p$--orthogonal} if $\langle e_1,e_2\rangle_p = 0$ for all distinct elements $e_1,e_2\in E$.
\item \emph{$p$--orthonormal} if $E$ is $p$--orthogonal and $p(e)=1$ for all $e\in E$.
\end{itemize}
In particular, a $p$--orthonormal set $E$ is said to be a \emph{complete $p-$orthonormal system} if $E$ is total in $E$, i.e. $\overline{span(E)}^p=V$. Such a system is also known as \emph{orthonormal basis}.

\begin{dfn}\thlabel{def::trace}
Let $p$ and $q$ be two Hilbertian seminorms on $V$. The \emph{trace of $p$ w.r.t.\ $q$} is denoted by $\mathrm{tr}(p/q)$ and defined as
$$
\mathrm{tr}(p/q):=
\begin{cases}
\sup\limits_{E\in\mathrm{FON}(q)}\,\,\sum\limits_{e\in E}p(e)^2, & \text{if }\ker(q)\subseteq\ker(p)\\
\infty, & \text{otherwise}
\end{cases},
$$
where $\mathrm{FON}(q)$ denotes the collection of all finite $q$--orthonormal subsets of $V$.
\end{dfn}

{When there exists $C>0$ such that $p\leq C q$} the following characterization of the trace of $p$ w.r.t.\! $q$ holds (by combining \thref{prop::trace-HS} and \eqref{eq::trace2} in Appendix \ref{sec: app-trace}):
\begin{equation}\label{char-trace}
\forall \text{$E$ complete $q-$orthonormal system in $V$},\,\,\mathrm{tr}(p/q)=
\sum\limits_{e\in E}p(e)^2.
\end{equation}

The following properties are immediate from the \thref{def::trace}.

\begin{lem}\thlabel{prop::trace}
Let $p$ and $q$ be two Hilbertian seminorms on $V$ with $\mathrm{tr}(p/q)<\infty$. Then:
\begin{enumerate}[label = (\roman*)]
	\item\label{properties-trace::1}
$p^2\leq\mathrm{tr}(p/q)q^2$.
	\item\label{properties-trace::2}
$\forall \varepsilon, \delta >0,\ \mathrm{tr}(\varepsilon p/\delta q)=\left(\frac{\varepsilon}{\delta}\right)^2\mathrm{tr}(p/q)$.
\item $\forall\ W\ \text{subspace of }\ V, \ \tr(p\!\restriction_W/ q\!\restriction_W)\leq\tr(p/q)$.
\end{enumerate}
\end{lem}

%\begin{proof}\ \\
%(i) Let $v\in V$. If $q(v)>0$, then $\{q(v)^{-1}v\}\subseteq V$ is $q-$orthonormal and so, $p(q(v)^{-1}v)\leq \sqrt{\tr(p/q)}$. If $q(v)=0$, then also $p(v)=0$ as $\ker(q)\subseteq\ker(p)$. In both cases the assertion is immediate.
%
%\noindent (ii) Let $\varepsilon,\delta>0$. Further, let $E\subseteq V$ be finite.  Then the set $E$ is orthonormal w.r.t.\ $q$ if and only if $\delta^{-1}E$ is orthonormal w.r.t.\ $\delta q$. In this case
%$$
%\sum_{e\in E}(\varepsilon p(\delta^{-1}e))^2=\varepsilon^2\delta^{-2}\sum_{e\in E}p(e)^2
%$$
%yields the assertion.
%
%\noindent (iii) is immediate as the collection of all finite $q$--orthonormal subsets of $W$ is contained in the the collection of all finite $q$--orthonormal subsets of $V$, being $W\subseteq V$.
%\end{proof}

We are equipped now with all notions needed to introduce the definition of a nuclear space due to Yamasaki (see \cite[Definition~20.1]{Ya85}), which we are going to adopt in this article.
\begin{dfn}
\thlabel{def::nuclear-space}
A locally convex space $(V,\tau)$ is called \emph{nuclear} if $\tau$ is induced by a directed family $\P$ of Hilbertian seminorms on $V$ such that for each $p\in\P$ there exists $q\in\P$ with $\mathrm{tr}(p/q)<\infty$.
\end{dfn}

\thref{def::nuclear-space} is equivalent to the more traditional ones in \cite{Gr55} and \cite{Mi61}, which we report in Appendix \ref{sec:app-nuclear} for the convenience of the reader (see \thref{def::nuclear-space-Grothendieck} and \thref{def::nuclear-space-Mityagin}).

Note that a nuclear topology can be always constructed on every vector space $V$. However, this nuclear topology has typically no relation with a pre-given topology $\tau_V$ on $V$. However, when $(V, \tau_V)$ is a separable locally convex space with a Schauder basis, there exists a dense subspace $U$ of $V$ on which a nuclear topology stronger than $\tau_V\restriction_U$ can be constructed.

\subsection{Probabilities on finite dimensional Hilbertian seminormed spaces.}\label{subs1.2}\ 

In the following we introduce a fundamental result about the support localization of a Radon measure defined on the dual of a finite dimensional real vector space, namely \thref{lem::fund-lem-Umemura}, which is inspired by \cite[Fundamental lemma (p.~24)]{Um65} and will play a crucial role in the proof of our main theorem \thref{MainThm}. For this, let us recall two properties of  the Gaussian measure on a finite dimensional real vector space endowed with a Hilbertian seminorm (see \cite[p.~26-28]{Um65} for a proof).

\begin{prop}\thlabel{prop-Gaussian-measure}
Let $q$ be a Hilbertian seminorm on an $n-$dimensional $\RR$--vector space $V$ with $\ker(q)=\{0\}$ and $E$ a complete $q-$orthonormal system of $V$. Let $\gamma$ be the Gaussian measure on $V$, i.e.,
$$
\dd\gamma(v) := (2\pi)^{-\tfrac{n}{2}}\exp\left(-\tfrac{1}{2}q(v)^2\right)\dd\lambda(v),
$$
where $\lambda$ is the measure on $V$ corresponding to the Lebesgue measure on $\RR^n$ under the identification $V\to\RR^n, v\mapsto (\langle v,e\rangle_q)_{e\in E}$. Then the following properties hold
\begin{enumerate}[label = (\roman*)]
	\item\label{prop-Gaussian-measure::3}
	$\int\langle v,w \rangle_q^2\dd\gamma(v)=1$ for all $w\in V$ such that $q(w)=1$.
		\item\label{prop-Gaussian-measure::4}
		$\gamma(\{v\in V: |\ell(v)|\geq 1\})\geq 7^{-1}$ for all $\ell\in V^\prime$ with $q'(\ell)\geq 1$, where $q'$ denotes the operator seminorm on $V^\prime$. \end{enumerate}
\end{prop}

\begin{lem}
\thlabel{lem::fund-lem-Umemura}
Let $p$ and $q$ Hilbertian seminorms on a finite dimensional $\RR$--vector space such that $\tr(p/q)<\infty$ and let $\mu$ be a probability measure on $V^\prime$.%($=V^\prime_q=V^\prime_p=V^\ast$).

\noindent If for any $\varepsilon>0$ there exists $\delta>0$ such that $\mu(\{l\in V^\prime:\left|l(v)\right|\geq 1\})\leq \varepsilon$ for all $v\in B_\delta(p)$, then
$$
\mu(B_1(q^\prime))\geq 1-7(\varepsilon+\tr(p/\delta q)),$$
where $q'$ denotes the operator seminorm on $V^\prime$.%, i.e. $\sup_{v\in B_1(q)}|\ell(v)|<\infty$.
\end{lem}
\begin{proof} Let $\mu$ be a probability measure on $V^\prime$.
W.l.o.g. we can assume $\ker(q)=\{0\}$, because otherwise for each complement $W$ of $\ker(q)$ in $V$ we have $\mu$--almost surely that $
B_1(q^\prime)=\{l\in V^\prime:\left|l(w)\right|\leq q(w)\ \forall\ w\in W\}
$ holds\footnote{On the one hand, it is immediate that
$B_1(q^\prime)=\{l\in V^\prime: q'(l)\leq 1\}=\{l\in V^\prime:\left|l(v)\right|\leq q(v)\text{ for all }v\in V\}\subseteq \{l\in V^\prime:\left|l(w)\right|\leq q(w)\text{ for all }w\in W\}$. On the other hand, since $p^2\leq\tr(p/q)q^2$, we have that $\ker(q)\subseteq\ker(p)\subset B_\delta(p)$ for all $\delta>0$ and so the assumption on $\mu$ in \thref{lem::fund-lem-Umemura} ensures that $\mu(\{l\in V^\prime:\left|l(v)\right|\geq 1\})=0$ for all $v\in\ker(q)$, i.e. $\mu(\{l\in V^\prime:\left|l(v)\right|=0, \ \forall {v}\in \ker(q)\})=1$, which immediately provides that $\mu\left\{B_1(q')\setminus \{l\in V^\prime:\left|l(w)\right|\leq q(w)\text{ for all }w\in W\}\right\}=0$.}. Consider
$$
D:=(\mu\times\gamma)(\{(l,v)\in V^\prime\times V:\left|l(v)\right|\geq 1\}),
$$
where $\mu\times\gamma$ denotes the product measure between the given measure $\mu$ on $V'$ and the Gaussian measure~$\gamma$ on $V$.
Now, let $l\in V^\prime\setminus B_1(q^\prime)$. Then Fubini's theorem on the one hand, combined with \thref{prop-Gaussian-measure}-\eqref{prop-Gaussian-measure::4}, provides that
$$
D =\int_{V^\prime}\gamma(\{v\in V:\left|l(v)\right|\geq 1\})\dd\mu(l)\geq 7^{-1}\mu(V^\prime\setminus B_1(q^\prime))=7^{-1}\left(1-\mu(B_1(q^\prime))\right),
$$
and, on the other hand, combined with the assumption yields that
\begin{equation}\label{est2}
D=\int_V\mu(\{l\in V^\prime: \left|l(v)\right|\geq 1\})\dd\gamma(v)\leq \varepsilon\gamma(B_\delta(p))+\gamma(V\setminus B_\delta(p)).
\end{equation}
Moreover, by~\cite[V, \S4.8, Theorem 2]{BouTVS}, there exists a complete $q-$orthonormal system $E$ of $V$ that is $p-$orthogonal. In particular,
$
p(v)^2=\sum\limits_{e\in E}\langle v,e\rangle_q^2p(e)^2$ holds for all $v\in V,
$
which combined with~\thref{prop-Gaussian-measure}-\eqref{prop-Gaussian-measure::3} and \eqref{char-trace} gives $$
\gamma(V\setminus B_\delta(p))\leq\delta^{-2}\int_Vp(v)^2\dd\gamma(v)=\delta^{-2}\sum_{e\in E}p(e)^2\int_V\langle v,e\rangle_q^2\dd\gamma(v)=\delta^{-2}\tr(p/q).
$$
The latter together with \eqref{est2} and~\thref{prop::trace}-\eqref{properties-trace::2} provides
\begin{equation}\label{est1}
D\leq \varepsilon+\delta^{-2}\tr(p/q)= \varepsilon+\tr(p/\delta q).
\end{equation}
Combining \eqref{est2} and \eqref{est1} yields the assertion.
\end{proof}

\section{Main results}\label{sec::main-results}

In this section we are going to present our main results concerning \ref{GenKMP} for a unital commutative real algebra $A$ generated by a vector space $V$ first endowed with a Hilbertian seminorm $q$ and then with a nuclear topology. More precisely, in Subsection \ref{sec:hilb-sem} we first establish a criterion for the existence of a representing measure with support contained in the set of characters of $A$ whose restrictions to $V$ are $q-$continuous (see \thref{MainThm} and \thref{rem-mainThm}, as well as \thref{MainThm-supp}). When the seminorm $q$ is defined on the full algebra $A$, i.e. $A=V$, this result provides in particular a criterion for the existence of a representing measure on the Gelfand spectrum of $q$. We actually show that when $L\leq C q$ on $A$ for some C then it is enough to check the latter criterion just on a dense subalgebra of $A$ (see \thref{{MainThm-supp_2}}). Moreover, in \thref{lemma::a} we provide an explicit bound on $L$ which guarantees the existence of a Hilbertian seminorm $q$ on $A$ satisfying our criteria.

Exploiting our general criteria, in \thref{thm-um-generating-subspace}, we identify more concrete sufficient conditions on $L$ and $q$ for the existence of such a representing measure for $L$. Those allow us to clarify in Subsection \ref{sec:nucl} the relation between the solvability of \ref{GenKMP} and the presence of a nuclear topology on $V$. Our general criteria are based on the projective limit approach introduced in \cite{InKuKuMi22} which allows to reduce \ref{GenKMP} to a family of finite-dimensional moment problems whose solutions satisfy a concentration condition to which we dedicate Subsection \ref{sec:conc}.
\\

\subsection{The concentration condition}\label{sec:conc}\ \\
Let $A$ be a unital commutative $\RR$-algebra generated by a linear subspace $V\subseteq A$ such that $X(A)\neq\emptyset$, and $L$ a normalized linear functional on $A$.

As already mentioned, in proving our main results for \ref{GenKMP} we will exploit the projective limite approach we developed in \cite{InKuKuMi22}. This is based on the construction of $(X(A), \tau_{X(A)})$ together with the maps $\{\pi_S: S\in J\}$ as the projective limit of the projective system of Hausdorff spaces $\{(X(S),\tau_{X(S)}),\pi_{S,T}, J\}$, where $$J:=\{\langle W\rangle: W \text{ finite\ dimensional\ subspace of }V\}$$ is ordered by inclusion, $\langle W\rangle$ denotes by the subalgebra of $A$ generated by~$W$, $\tau_{X(S)}$ is the weak topology on $X(S)$, for any $S,T$ subalgebras of $A$ with $S\subseteq T$ the map $\pi_{S,T}: X(T)\to X(S)$ is the natural restriction and $\pi_S:=\pi_{S, A}$. The corresponding projective system of measurable spaces is given by $\{(X(S), \B(\tau_{X(S)})), \pi_{S, T}, J\}$,  where $\B(\tau_{X(S)})$ is the Borel $\sigma-$algebra w.r.t. $\tau_{X(S)}$. Recall that this means that $\pi_{S, T}$ is measurable for all $S\subseteq T$ in $J$ and that $\pi_{S, T}\circ\pi_{T, R}$ for all $S\subseteq T\subseteq R$ in $J$.

Roughly speaking, in \cite{InKuKuMi22}, we establish that there exists a representing Radon measure for $L$ on $A$ supported in $(X(A), \B(\tau_{X(A)}))$ if and only if for each $S\in J$ there exists a representing Radon measure $\nu_S$ supported in $(X(S), \B(\tau_{X(S)}))$ such that $\{\nu_S: \ S\in J\}$ fulfills the so-called Prokhorov condition.
 In the next subsection we will exploit this result when studying \ref{GenKMP} for $V$ endowed with the topology $\tau_V$ induced by a Hilbertian seminorm and we will exploit the given topological structure on $V$ to prove that the Prokhorov condition (see \cite[Section 1.2]{InKuKuMi22} and references therein) is satisfied whenever $\{\nu_S: \ S\in J\}$ fulfills the following concentration property.

\begin{dfn}\label{cont-projsys}
Given a seminorm $p$ on $V$ and for each $S\in J$ a Radon measure $\nu_S$ on $(X(S), \B(\tau_{X(S)}))$, we say that $\{\nu_S: S\in J\}$ is \emph{$p-$concentrated} (or \emph{concentrated w.r.t. $p$}) if
\begin{equation}\label{eq::continuous}
\forall\varepsilon>0\,\exists\delta>0\colon \forall S\in J, \forall a\in B_{\delta}(p)\cap S, \nu_S(\{\alpha\in X(S):\left|\alpha(a)\right|\geq 1\})\leq \varepsilon.
\end{equation}
\end{dfn}

This definition is an adaptation to our setting of the notion of \emph{continuity for cylindrical measures} introduced in \cite[16, Chapter IV, Section 1.4]{GV64}. It also easily relates to the notion of \emph{concentrations of cylindrical measures} in \cite[Definition 1, p.192]{S73}. In fact, \eqref{eq::continuous} is weaker than assuming that the cylindrical quasi-measure associated to $\{\nu_S:S\in J\}$ is cylindrically concentrated on $\{\mathfrak{sp}_C(p): C>0\}$, namely $\forall\varepsilon>0\,\exists \delta>0 \colon \forall S\in J,$  $\nu_S(\pi_S(\mathfrak{sp}_{\delta}(p)))\geq 1-\varepsilon$.\\ %, where $\pi_S(\alpha):=\alpha\restriction_S$ for all $\alpha\in X(A)$.

Let us now provide a useful characterization of the $p-$concentration of a collection of Radon measures.

\begin{prop}\thlabel{char-cont-projsys}
Given a seminorm $p$ on $V$ and for each $S\in J$ a Radon measure $\nu_S$ on $(X(S), \B(\tau_{X(S)}))$, we have that $\{\nu_S: S\in J\}$ is $p-$concentrated if and only if the following holds
\begin{equation}\label{eq::continuous2}
\forall\varepsilon>0\,\exists\gamma>0\colon \forall S\in J, \forall a\in S\cap V, \nu_S(\{\alpha\in X(S):\left|\alpha(a)\right|\leq \gamma p(a)\})\geq 1-\varepsilon.
\end{equation}

\end{prop}

\proof \ \\
Suppose \eqref{eq::continuous} holds and fix $\varepsilon>0$. Taking $0<\delta'<\delta$ with $\delta$ as in \eqref{eq::continuous}, we have that \eqref{eq::continuous2} holds for $\gamma=\frac{1}{\delta'}$. In fact, for any $S\in J$, let $b\in S\cap V$ and distinguish the following two cases.
 \begin{itemize}[leftmargin=*]
\item If $p(b)\neq 0$, then $\frac{\delta' b}{p(b)}\in B_{\delta}(p)\cap S$ and so \eqref{eq::continuous} provides that $\nu_S(\{\alpha\in X(S):\left|\alpha(b)\right|\geq \frac{p(b)}{\delta}\})\leq \varepsilon$, which implies $\nu_S(\{\alpha\in X(S):\left|\alpha(b)\right|\leq \frac{p(b)}{\delta'}\})\geq 1-\varepsilon$.

\item If $p(b)= 0$, then clearly $span(b)\subseteq B_\delta(p)\cap V\cap S$ and so \eqref{eq::continuous} gives that
$\forall \lambda>0,$ $\forall a\in span(b), \ \nu_S\left(\{ \alpha \in X(S) \, : \, |\alpha(a)| \geq 1 \}\right)\leq \lambda,$
i.e. $\forall a\in span(b),$ $\ \nu_S\left(\{ \alpha \in X(S) \, : \, |\alpha(a)| < 1 \}\right)=1.$
Then
$$\forall r>0, \ \nu_S\left(\left\{ \alpha \in X(S) \, : \, |\alpha(b)| < \frac{1}{r} \right\}\right)=1,$$
and so we get
$\nu_S(\{ \alpha \in X(S) \, : \, |\alpha(b)| = 0 \})=1$, which in particular gives that $\nu_S(\{ \alpha \in X(S) \, : \, |\alpha(b)| \leq \frac{p(b)}{\delta'}\})=1\geq 1-\varepsilon $.
\end{itemize}

\indent Conversely, suppose \eqref{eq::continuous2} holds and fix $\varepsilon>0$. Taking $\gamma$ as in \eqref{eq::continuous2}, we have that \eqref{eq::continuous} holds for $\delta\leq\frac{1}{\gamma}$. In fact, for any $S\in J$, let $b\in B_{\delta}(p)\cap S$ and distinguish the following two cases.
 \begin{itemize}[leftmargin=*]
\item If $p(b)\neq 0$, then \eqref{eq::continuous2} provides that $\nu_S(\{\alpha\in X(S):\left|\alpha(b)\right|\leq \gamma p(b) \})\geq 1-\varepsilon$ which implies $\nu_S(\{\alpha\in X(S):\left|\alpha(b)\right|<1\})\geq 1- \varepsilon$.%, i.e. $\nu_S(\{\alpha\in X(S):\left|\alpha(b)\right|\geq 1\})\leq \varepsilon$.

\item If $p(b)= 0$,  then \eqref{eq::continuous2} provides that $\nu_S(\{\alpha\in X(S):\left|\alpha(b)\right|=0\})\geq 1-\varepsilon$, i.e. $\nu_S(\{\alpha\in X(S):\left|\alpha(b)\right|>0\})\leq \varepsilon$, which implies
$\nu_S(\{\alpha\in X(S):\left|\alpha(b)\right|\geq 1\})\leq \varepsilon$.\end{itemize}
 \endproof

\begin{rem}\thlabel{vanish-on-ker}\ \\
If $\nu$ is a Radon measure on $X(A)$ s.t. $\{{\pi_S}_{\#}\nu :S\in J\}$ is $p-$concentrated, then
\begin{equation}\label{vanish-on-kernel}
\nu(\{ \alpha \in X(A) \, : \, |\alpha(b)| = 0\ \forall b \in \ker(p) \}) =1,
\end{equation}
where ${\pi_S}_{\#}\nu$ denotes the pushforward measure of $\nu$ w.r.t.\! $\pi_S$.
\noindent Indeed, using the same argument as in the proof of \thref{char-cont-projsys}, we can show that  %if $b\in \ker(p)$ and $S=\langle b\rangle$ then $span(b)\subseteq B_\delta(p)\cap V\cap S $ for all $\delta >0$. Therefore, \eqref{eq::continuous} provides that
%$$\forall \varepsilon>0, \forall a\in span(b), \ {\pi_S}_{\#}\nu\left(\{ \alpha \in X(S) \, : \, |\alpha(a)| \geq 1 \}\right)\leq \varepsilon,$$
%i.e. $\forall a\in span(b), \ {\pi_S}_{\#}\nu\left(\{ \alpha \in X(S) \, : \, |\alpha(a)| < 1 \}\right)=1.$
%Then
%$$\forall r>0, \ {\pi_S}_{\#}\nu\left(\left\{ \alpha \in X(S) \, : \, |\alpha(b)| < \frac{1}{r} \right\}\right)=1,$$
%and so
$
 \forall b\in \ker(p), \nu(\{ \alpha \in X(A) \, : \, |\alpha(b)| = 0 \})={\pi_{\langle b\rangle}}_{\#}\nu(\{ \alpha \in X(\langle b\rangle) \, : \, |\alpha(b)| = 0 \}) = 1$.
This together with the fact that $\{ \alpha \in X(A) \, : \, |\alpha(b)| = 0 \}$ is a closed subset of~$X(A)$ and $\nu$ a Radon measure yields \eqref{vanish-on-kernel} by \cite[Part I, Chapter I, 6.(a)]{S73}.
\end{rem}

Let us establish now a sufficient condition for the $p-$concentration of a collection of representing measures, which we will often exploit in the rest of the article.
\begin{lem}\label{lem-suff-conc}
Let $A$ be an algebra generated by a linear subspace $V$ and $p$ a seminorm on $V$. Given a  linear functional $L$ on $A$ such that $L(\sum A^2)\subseteq [0,\infty)$ and, for each $S\in J$, a representing measure $\nu_S$ for $L\restriction_S$, if
\begin{equation}\label{con-L}
\exists\ C>0 \, : \, L(a^2) \leq C p(a)^2\qquad\text{for all }a\in V
\end{equation}
holds, then $\{\nu_S:S\in J\}$ is $p-$concentrated.
\end{lem}
\begin{proof} Let $\varepsilon>0$ and take $\delta:= \sqrt{\frac{\varepsilon}{C}}$. Then for all  $a\in B_\delta(p)\cap S$
$$ \nu_S(\{\alpha\in X(S):\left|\alpha(a)\right|\geq 1\})\leq \int_{X(S)} \hat{a}^2 d\nu_S=L(a^2)\leq Cp(a)^2\leq C\delta^2\leq \varepsilon$$
i.e. $\{\nu_S:S\in J\}$ fulfills \eqref{eq::continuous}.
\end{proof}
With a similar proof one gets the following generalization of \eqref{con-L}
\begin{equation}\label{con-L-Patrick}
\forall\varepsilon>0\ \exists\ C>0 : L(a^2)\leq Cp(a)^2 +\varepsilon\qquad\text{for all }a\in V.
\end{equation}

\subsection{The case when $\tau_V$ generated by a Hilbertian seminorm}\label{sec:hilb-sem}

\begin{thm}
\thlabel{MainThm}
Let $A$ be an algebra generated by a linear subspace $V\subseteq A$, $q$ a Hilbertian seminorm on $V$ such that $\{\alpha\in X(A):\alpha\!\restriction_V\text{is }q\text{--continuous}\}\neq\emptyset$ and $J:=\{\langle W\rangle: W \text{ finite\ dimensional\ subspace of }V\}.$ Let $L$ be a normalized linear functional on $A$.

There exists a representing Radon measure $\nu$ for $L$ with support contained in $\{\alpha\in X(A):\alpha\!\restriction_V\text{is }q\text{--continuous}\}$ if and only if there exists a Hilbertian seminorm $p$ on $V$ with $\tr(p/q)<\infty$ and for each $S\in J$ there exists a representing Radon measure $\nu_S$ for $L\!\restriction_S$ with support contained in $X(S)$ and such that $\{\nu_S:S\in J\}$ is $p-$concentrated.
\end{thm}
\begin{proof}
For each $S\in J$, let $\nu_S$ be a representing Radon measure for $L\!\restriction_S$ with support contained in $X(S)$ and $p$ be a Hilbertian seminorm $p$ on $V$ with $\tr(p/q)<\infty$ such that $\{\nu_S:S\in J\}$ is $p-$concentrated.
Let us first show that the family $\{\nu_S:S\in J\}$ fulfils the so-called Prokhorov condition by means of the characterization in \cite[Proposition~1.18]{InKuKuMi22}, that is, we aim to show that for all $\varepsilon>0$ and for all $S\in J,$ there exists $K^{(S)}\subseteq X(S)$ compact such that $ \nu_S(K^{(S)})\geq 1-\varepsilon$ and $\pi_{S,T}(K^{(T)})\subseteq K^{(S)}$ for all $T\in J$ with $S\subseteq T$.

Let $\varepsilon>0$. Since $\{\nu_S:S\in J\}$ is $p-$concentrated and $\tr(p/q)<\infty$, we can take $\delta>0$ as in~\eqref{eq::continuous} and set $r_\varepsilon:=q(\delta\sqrt{\varepsilon})^{-1}\sqrt{\tr(p/q)}$. For each $S\in J$, define $K^{(S)}:=\{\alpha\in X(S):\left|\alpha(v)\right|\leq r_\varepsilon(v)\text{ for all }v\in S\cap V\}$. Then $K^{(S)}$ is compact in $X(S)$ as it is closed and embeds into the compact product $\prod_{v\in S\cap V}[-r_\varepsilon(v),r_\varepsilon(v)]$ via the continuous map $\alpha\mapsto(\alpha(v))_{v\in S\cap V}$.
Now for any $S\subseteq T$ in $J$ the inclusion $\pi_{S,T}(K^{(T)})\subseteq K^{(S)}$ holds by definition and for each $S\in J$ the estimate $\nu_S(K^{(S)})\geq 1-14\varepsilon$ holds by~\thref{fundamental-lemma-um} below. Hence, the family $\{\nu_S:S\in J\}$ fulfils Prokhorov's condition and so we can apply~\cite[Theorem~3.10-(ii)]{InKuKuMi22}, which guaranteed the existence of a representing Radon measure~$\nu$ for~$L$ with support contained in $X(A)$. It remains to show that the support of $\nu$ is contained in $\{\alpha\in X(A):\alpha\!\restriction_V\text{ is }q\text{--continuous}\}$. For this, set $\phi_{V}: X(A)\to V^*, \alpha\mapsto \alpha\restriction_V$ and
$$
K_\varepsilon :=\{\alpha\in X(A):\left|\alpha(v)\right|\leq r_\varepsilon(v)\text{ for all }v\in V\} = \bigcap_{S\in J}\pi_S^{-1}(K^{(S)}).
$$
Then \cite[Propositions~7.2.2-(i) and~7.2.5-(iii)]{Bog07b} and \thref{fundamental-lemma-um} imply that $$\nu(K_\varepsilon)=\lim_{S\in I}\nu_S(K^{(S)})\geq 1-14\varepsilon.$$
Since $K_\varepsilon\subseteq{\phi_V}^{-1}(V_{r_\varepsilon}^\prime)={\phi_V}^{-1}(V_q^\prime)$ for all $\varepsilon>0$ this yields that $\nu({\phi_V}^{-1}(V_q^\prime))=1$, i.e. $\nu$ has support contained in ${\phi_V}^{-1}(V_q^\prime)=\{\alpha\in X(A):\alpha\!\restriction_V\text{ is }q\text{--continuous}\}$.

Conversely, let $\nu$ be a representing Radon measure for $L$ with support contained in $\{\alpha\in X(A):\alpha\!\restriction_V\text{ is }q\text{--continuous}\}$.
Then, for each $S\in J$, the push-forward $\nu_S:={\pi_S}_{\#}\nu$ is a representing Radon measure for $L\!\restriction_S$ with support contained in~$X(S)$. %by~\cite[Theorem~3.10-(ii)]{InKuKuMi22}.
For each $n\in\NN$, set $K_n:={\phi_V}^{-1}(B_n(q^\prime))$ and define \begin{equation}\label{thm::Sazonov::eq1}
p(v)^2:=\sum_{n=1}^\infty\frac{1}{n^4}\int_{K_n}\hat{v}^2\dd\nu\qquad\text{for all }v\in V.
\end{equation}
It is easy to verify that $p$ defines a Hilbertian seminorm on $V$. Then
for each $E\in\mathrm{FON}(q)$ we have that
$$
\sum_{e\in E}p(e)^2
\stackrel{\eqref{thm::Sazonov::eq1}}{=}\sum_{n=1}^\infty\frac{1}{n^4}\int_{K_n}\sum_{e\in E}\hat{e}^2\dd\nu
\stackrel{\text{Lemma}\ \ref{thm::Sazonov::lem2}}{\leq} \sum_{n=1}^\infty\frac{1}{n^4}\int_{K_n}n^2\dd\nu
\leq\sum_{n=1}^\infty\frac{1}{n^2}
<2,
$$
that is, $\tr(p/q)\leq 2<\infty$.
To show that $\{\nu_S:S\in J\}$ is $p-$concentrated, let $\varepsilon>0$
and take $n\in\NN$ such that $\nu(X(A)\setminus K_n)\leq 2^{-1}\varepsilon$. Then there exists $\delta>0$ such that $n^4\delta^2\leq 2^{-1}\varepsilon$ and so, for each $S\in J$ and each $v\in B_\delta(p)\cap S$, we obtain that
\begin{eqnarray*}
\nu_S(\{\alpha\in X(S):\left|\alpha(v)\right|\geq 1\}) &\leq& \nu(\{\alpha\in X(A):\left|\alpha(v)\right|\geq 1\}\cap K_n) + \nu(X(A)\setminus K_n)\\
&\leq& \int_{K_n}\hat{v}^2\dd\nu_S+2^{-1}\varepsilon
\leq n^4p(v)^2+2^{-1}\varepsilon
\leq \varepsilon,
\end{eqnarray*}
i.e.~\eqref{eq::continuous} holds.
\end{proof}

\begin{lem}
\thlabel{fundamental-lemma-um}

For each $S\in J$, the estimate $\nu_S(K^{(S)})\geq 1-14\varepsilon$ holds.
\end{lem}
\begin{proof}
Let $S\in J$ and set $I:=\{W\subseteq S\cap V: W\text{ fin.\ dim.\ subspace of } S\cap V\}$.
Let $W\in I$ and consider the continuous restriction map $\phi_W\colon X(S)\to W^\prime$. Then the push-forward $\mu_W^\prime:={\phi_W}_{\#}\nu_S$ is a probability measure on $W^\prime$ that satisfies
$$
\mu_W^\prime(\{l\in W^\prime:\left| l(w)\right|\geq 1\}) = \nu_S(\{\alpha\in X(S):\left|\alpha(w)\right|\geq 1\})\leq \varepsilon
$$
for all $w\in B_\delta(p\!\restriction_W)$.
Then \thref{lem::fund-lem-Umemura} implies that
$$
\nu_S({\phi_W}^{-1}(B_1(r_\varepsilon\!\restriction_W^\prime)))=\mu_W^\prime(B_1(r_\varepsilon\!\restriction_W^\prime)))\geq 1-7(\varepsilon+\tr(p\!\restriction_W/\delta r\!\restriction_W))=1-14\varepsilon
$$
as $\tr(p\!\restriction_W/\delta r_\varepsilon\!\restriction_W)\leq\tr(p/\delta r_\varepsilon)$ and, by \thref{prop::trace}-\eqref{properties-trace::2}, $\tr(p/\delta r_\varepsilon)\leq\varepsilon$.

Since $K^{(S)}=\bigcap_{W\in I} {\phi_W}^{-1}(B_1(r_\varepsilon\!\restriction_{W^\prime}))$ by definition, \cite[Propositions~7.2.2-(i) and~7.2.5-(iii)]{Bog07b} imply that
$
\nu_S(K^{(S)})=\lim_{W\in I}\nu_S({\phi_W}^{-1}(B_1(r\!\restriction_{W^\prime})))\geq 1-14\varepsilon.
$
\end{proof}

\begin{lem}
\thlabel{thm::Sazonov::lem2}
Let $n\in\NN$ and $E\in\mathrm{FON}(q)$. Then $\sum\limits_{e\in E}\hat{e}(\alpha)\leq n^2$ for all $\alpha\in K_n$.
\end{lem}
\begin{proof}
Let $\alpha\in K_n$ and set $H:=\mathrm{span}(E)$ (for convenience, set $\alpha=\alpha\!\restriction_H$ and $q=q\!\restriction_H$). Since $E\in\mathrm{FON}(q)$ is finite, the space $(H,q)$ is Hilbertian and $E$ is a complete $q-$orthonormal system.
In particular, $\alpha\!\restriction_H\in B_n(q\!\restriction_H^\prime)$ and by the Riesz representation theorem there exists $a\in H$ such that $\alpha(x)=\langle x,a\rangle_q$ for all $x\in H$ and $q(a)=q^\prime(\alpha)\leq n$.
Therefore,
$$
\sum_{e\in E}\hat{e}(\alpha)^2=\sum_{e\in E}\alpha(e)^2=\sum_{e\in E}\langle e,a\rangle_q^2=q(a)^2\leq n^2
$$
yields the assertion.
\end{proof}

Using exactly the same proof scheme but exploiting \cite[Corollary~3.11-(ii)]{InKuKuMi22} instead of \cite[Theorem~3.10-(ii)]{InKuKuMi22}, it is easy to obtain the following more general version of \thref{MainThm} including the localization of the support of the representing measure.

\begin{thm}\thlabel{MainThm-supp}\ \\
Let $A$ be an algebra generated by a linear subspace $V\subseteq A$, $K\subseteq X(A)$ closed, $q$ a Hilbertian seminorm on $V$ such that $\{\alpha\in K:\alpha\!\restriction_V\text{is }q\text{--continuous}\}\neq\emptyset$ and $J:=\{\langle W\rangle: W \text{ finite\ dimensional\ subspace of }V\}.$ Let $L$ be a normalized linear functional on $A$ and $Q$ a quadratic module such that $K=K_Q$.

There exists a representing Radon measure $\nu$ for $L$ with support contained in $\{\alpha\in K: \alpha\!\restriction_V\text{is }q\text{--continuous}\}$ if and only if there exists a Hilbertian seminorm $p$ on $V$ with $\tr(p/q)<\infty$ and for each $S\in J$  there exists a representing Radon measure $\nu_S$ for $L\!\restriction_S$ with support contained in $K_{Q\cap S}$ and such that $\{\nu_S:S\in J\}$ is $p-$concentrated.
\end{thm}

\begin{rem}\thlabel{rem-mainThm}\
\begin{enumerate}
\item If in \thref{MainThm} (resp.\! \thref{MainThm-supp}) we assume for each $S\in J$ the \emph{uniqueness} of the representing measure for $L\!\restriction_S$ with support contained in $X(S)$ (resp. in $K_{Q\cap S}$) , then by \cite[Remark~3.12-(ii)]{InKuKuMi22} we get the uniqueness of the corresponding representing measure for $L$.
\item If in \thref{MainThm} (resp. \thref{MainThm-supp}) we take $V=A$, we obtain a criterion for the existence of a representing measure for $L$ with support contained in $\sp(q)$ (resp. on $K_Q\cap\sp(q)$).%$\phi_V^{-1}(V^\prime_q)=\sp(q)$.
\item Combining \thref{MainThm} (resp. \thref{MainThm-supp}) and \thref{vanish-on-ker}, it is easy to see that if there exists a representing measure for $L$ with support contained in $\{\alpha\in X(A): \alpha\restriction_V \text{ is $q-$continuous}\}$ then $L$ vanishes on $\ker(p)$ and so on $\ker(q)$.
\end{enumerate}
\end{rem}

When $A$ is endowed with a Hilbertian seminorm $q$ and there exists $C>0$ such that $L(a^2)\leq C q(a)^2$ for all $a\in A$, we can characterize the representing measures for $L$ with support contained in $\sp(q)$ only through conditions on a dense subalgebra of $A$.

\begin{thm}\thlabel{MainThm-supp_2}
Let $A$ be an algebra, $q$ a Hilbertian seminorm on $A$ with $\sp(q)\neq\emptyset$, $B$ a subalgebra of $A$ which is dense in $(A, q)$ and $I:=\{\langle W\rangle: W \text{ finite\ dimen\-sional}$ $\text{subspace of }B\}.$ Let $L$ be a normalized $\sum A^2$--positive linear functional on $A$ for which there exists $C>0$ such that $L(a^2)\leq C q(a)^2$ for all $a\in A$.

There exists a representing Radon measure $\nu$ for $L$ with support contained in $\sp(q)$ if and only if there exists a Hilbertian seminorm $p$ on $B$ with $\tr(p/q)<\infty$ and for each $S\in I$  there exists a representing Radon measure $\nu_S$ for $L\!\restriction_S$ with support contained in $X(S)$ such that $\{\nu_S:S\in I\}$ is $p-$concentrated.

\end{thm}

\begin{proof}
By the density of $B$ in $(A, q)$, there is a one-to-one correspondence between the set of all $q-$continuous characters of $A$ and the set of all $q-$continuous characters of $B$, which will therefore both denote simply by $\sp(q)$. Moreover, let $\tau_{\sp(q)^A}$ (resp. $\tau_{\sp(q)^B}$) be the weakest topology on $\sp(q)$ which makes $\hat{a}:\sp(q)\to\RR, \alpha \mapsto \alpha(a)$ continuous for all $a \in A$ (resp. for all $a\in B$) and by $\mathcal{B}(\tau_{\sp(q)^A})$ (resp. $\mathcal{B}(\tau_{\sp(q)^B})$) the associated Borel $\sigma$-algebra. We demand to Appendix \ref{aux} the proof that
\begin{equation}\label{eq-sigma-algebras}
\mathcal{B}(\tau_{\sp(q)^A})= \mathcal{B}(\tau_{\sp(q)^B}).
\end{equation}
and so we will not distinguish between the measurable spaces $(\sp(q), \mathcal{B}(\tau_{\sp(q)^A}))$ and $(\sp(q), \mathcal{B}(\tau_{\sp(q)^B}))$, which will be both simply denoted by $(\sp(q), \mathcal{B}(\tau_{\sp(q)}))$.

Suppose that there exists a representing Radon measure $\nu$ for $L$ with support contained in $\sp(q)$. Then, applying \thref{MainThm} for $V=A=B$, we get that there exists a Hilbertian seminorm $p$ on $B$ with $\tr(p/q)<\infty$ and for each $S\in I$  there exists a representing Radon measure $\nu_S$ for $L\!\restriction_S$ with support contained in $X(S)$ such that $\{\nu_S:S\in I\}$ is $p-$concentrated.

Conversely, suppose there exists a Hilbertian seminorm $p$ on $B$ with $\tr(p/q)<\infty$ and for each $S\in I$ there exists a representing Radon measure $\nu_S$ for $L\!\restriction_S$ with support contained in $X(S)$ such that $\{\nu_S:S\in I\}$ is $p-$concentrated. Then, applying \thref{MainThm} for $V=A=B$ (see also \thref{rem-mainThm}-(2)), we obtain that there exists a representing measure $\nu$ for $L\restriction_B$ with support contained in $\sp(q)$. We aim to prove that $\nu$ is actually a representing measure for $L$, so it remains to show that
$L(a)=\int \hat{a}(\alpha) d\nu (\alpha), \ \forall a\in A\setminus B$.

Let $a \in A\setminus B$. By the density of $B$ in $(A, q)$, there exists a sequence $(b_n)_{n\in \NN}\subseteq B$ such that $q(b_n-a)\rightarrow 0$ as $n\to\infty$. Hence, for any $\alpha\in\sp(q)$ we get $\alpha(b_n) \rightarrow \alpha(a)$ as $n\to\infty$, i.e. $\lim_{n\to\infty}\hat{b_n}(\alpha)=\hat{a}(\alpha)$. Then, using Fatou's lemma, we obtain that

$$\int_{\sp(q)} \hat{a}(\alpha)^2 d\nu=\int_{\sp(q)} \lim_{n\to\infty}\hat{b_n}(\alpha)^2 d\nu\leq \liminf_{n\to\infty}\int_{\sp(q)} \hat{b_n}(\alpha)^2 d\nu=\liminf_{n\to\infty} L(b_n^2)=L(a^2),$$
where in the last equality we used the Cauchy--Bunyakovsky--Schwarz inequality and the existence of $C>0$ such that $L(a^2)\leq C q(a)^2$ for all $a\in A$. Hence, $\hat{a}\in \mathcal{L}^2(\sp(q), \B(\tau_{\sp(q)}))$ and so $\hat{a}\in \mathcal{L}^1(\sp(q),  \B(\tau_{\sp(q)}))$.

Then for any $M>0$ we have that
\begin{align}\label{passaggio}
\int_{\sp(q)}| \hat{a}(\alpha) - \hat{b_n}(\alpha)| d\nu(\alpha)
& \leq \int_{\sp(q)}\left| \hat{a}(\alpha)\ind_{\{\beta: |\hat{a}(\beta)|\leq M\}}(\alpha) - \hat{b}_n(\alpha)\ind_{\{\beta: |\hat{b}_n(\beta)|\leq M\}}(\alpha)\right| d\nu(\alpha) \nonumber\\
&\quad+\int_{\sp(q)}\left| \hat{a}(\alpha)\ind_{\{\beta: |\hat{a}(\beta)|\leq M\}}(\alpha) - \hat{a}(\alpha)\right| d\nu(\alpha) \nonumber\\
&\quad+\int_{\sp(q)}\left| \hat{b_n}(\alpha)\ind_{\{\beta: |\hat{b}_n(\beta)|\leq M\}}(\alpha) - \hat{b}_n(\alpha)\right| d\nu(\alpha)\nonumber\\
& =\int_{\sp(q)}\left| \hat{a}(\alpha)\ind_{\{\beta: |\hat{a}(\beta)|\leq M\}}(\alpha) - \hat{b}_n(\alpha)\ind_{\{\beta: |\hat{b}_n(\beta)|\leq M\}}(\alpha)\right| d\nu(\alpha) \nonumber\\
&\quad+\int_{\sp(q)}\left| \hat{a}(\alpha)\right|\ind_{\{\beta: |\hat{a}(\beta)|> M\}}(\alpha) d\nu(\alpha) \nonumber\\
&\quad+\int_{\sp(q)}\left| \hat{b}_n(\alpha)\right|\ind_{\{\beta: |\hat{b}_n(\beta)|> M\}}(\alpha) d\nu(\alpha)
\end{align}

Using that $\ind_{\{\beta: |\hat{b}_n(\beta)|> M\}}(\alpha)\leq \frac{1}{M}|\hat{b}_n(\alpha)|$ and that $\nu$ is a representing measure for $L\restriction_B$, we easily see that:
$$\int_{\sp(q)}\left| \hat{b}_n(\alpha)\right|\ind_{\{\beta: |\hat{b}_n(\beta)|> M\}}(\alpha) \leq\frac{1}{M}\int_{\sp(q)}\hat{b}_n(\alpha)^2 d\nu(\alpha)=\frac{L(b_n^2)}{M}\to \frac{L(a^2)}{M}, \text{ as } n\to\infty.$$

Therefore, passing to the limit for $n\to\infty$ in \eqref{passaggio}, we get that for any $M>0$:
\begin{equation}\label{pass-intermedio}
\lim_{n\to\infty}\int_{\sp(q)}| \hat{a}(\alpha) - \hat{b_n}(\alpha)| d\nu(\alpha)
\leq \int_{\sp(q)}\left| \hat{a}(\alpha)\right|\ind_{\{\beta: |\hat{a}(\beta)|> M\}}(\alpha) d\nu(\alpha)+\frac{L(a^2)}{M}
\end{equation}

Since $\hat{a}\in \mathcal{L}^1(\sp(q),  \B(\tau_{\sp(q)}))$ and $\left| \hat{a}(\alpha)\right|\ind_{\{\beta: |\hat{a}(\beta)|> M\}}(\alpha)  \leq \left| \hat{a}(\alpha)\right|$ for all $M>0$, we can apply the dominated converge theorem, which ensures that:
$$\lim_{M\to\infty}\int_{\sp(q)}\left| \hat{a}(\alpha)\right|\ind_{\{\beta: |\hat{a}(\beta)|> M\}}(\alpha) d\nu(\alpha)=\int_{\sp(q)}\lim_{M\to\infty}\left| \hat{a}(\alpha)\right|\ind_{\{\beta: |\hat{a}(\beta)|> M\}}(\alpha) d\nu(\alpha)=0$$

Hence, passing to the limit for $M\to\infty$ in \eqref{pass-intermedio}, we obtain that:
 $$\lim_{n\to\infty}\int_{\sp(q)}| \hat{a}(\alpha) - \hat{b_n}(\alpha)| d\nu(\alpha) =0$$
 and so
 $$L(a)=\lim_{n\to\infty}L(b_n)=\lim_{n\to\infty}\int_{\sp(q)}\hat{b_n}(\alpha)d\nu(\alpha) =\int_{\sp(q)} \hat{a}(\alpha) d\nu(\alpha).$$

\end{proof}

With a similar proof, it is possible to show the following more general version of \thref{MainThm-supp_2}.
\begin{thm}\thlabel{MainThm-dense-supp}
Let $A$ be an algebra, $K\subseteq X(A)$ closed, $q$ a Hilbertian seminorm on $A$ with $\sp(q)\cap K\neq\emptyset$, $B$ a subalgebra of $A$ which is dense in $(A, q)$ and $I:=\{\langle W\rangle: W \text{ finite\ dimen\-sional}$ $\text{subspace of }B\}.$ Let $L$ be a normalized $\sum A^2$--positive linear functional on $A$ for which there exists $C>0$ such that $L(a^2)\leq C q(a)^2$ for all $a\in A$ and $Q$ a quadratic module such that $K=K_Q$.

There exists a representing Radon measure $\nu$ for $L$ with support contained in $\sp(q)\cap K$ if and only if there exists a Hilbertian seminorm $p$ on $B$ with $\tr(p/q)<\infty$ and for each $S\in I$  there exists a representing Radon measure $\nu_S$ for $L\!\restriction_S$ with support contained in $K_{Q\cap S}$ such that $\{\nu_S:S\in I\}$ is $p-$concentrated.
\end{thm}

Given a normalized $\sum A^2$--positive linear functional $L$, the map $(a,b)\mapsto L(ab)$ defines a symmetric positive semidefinite bilinear form and so the following is a natural Hilbertian seminorm on $A$
\begin{equation}\label{seminorm-pL}
s_L(a):=\sqrt{L(a^2)}\,\text{ for all}\ a\in A.
\end{equation}
Then, combining Lemma \ref{lem-suff-conc} with our main results \thref{MainThm-supp} and \thref{MainThm-dense-supp}, we easily obtain the following two results.

\begin{cor}\thlabel{theorem-main-s_L}
Let $A$ be an algebra generated by a linear subspace $V\subseteq A$, $K\subseteq X(A)$ closed and $J:=\{\langle W\rangle: W \text{ finite\ dimensional\ subspace of }V\}.$ Let $L$ be a normalized $\sum A^2$--positive linear functional $L$ on $A$ and $Q$ a quadratic module such that $K=K_Q$.

If there exists a Hilbertian seminorm $q$ on $V$ such that $\tr(s_L\restriction_V/q)<\infty$ and $\{\alpha\in K:\alpha\!\restriction_V\text{is }q\text{--continuous}\}\neq\emptyset$ and for each $S\in J$ there exists a representing Radon measure $\nu_S$ for $L\!\restriction_S$ with support contained in $K_{Q\cap S}$, then there exists a representing measure for $L$ with support contained in $\{\alpha\in K:\alpha\!\restriction_V\text{is }q\text{--continuous}\}$.
\end{cor}

\proof

Since $L(a^2)=s_L(a)^2$ for all $a\in V$, \eqref{con-L} holds for $p=s_L$ and $C=1$. Hence, Lemma \ref{lem-suff-conc} ensures that $\{\nu_S:S\in I\}$ is $s_L-$concentrated.
%Let $\varepsilon>0$ and $\delta:=\sqrt{\varepsilon}$. Then for each $S\in I$ and any $ a\in B_\delta(s_L)\cap S
%$ we have
%$$ \mu_S(\{\alpha\in X(S):\left|\alpha(a)\right|\geq 1\})\leq \int_{K_{Q\cap S}} \hat{a}^2 d\mu_S=L(a^2)=s_L(a)^2\leq\varepsilon$$
%i.e. $\{\mu_S:S\in I\}$ fulfills \eqref{eq::continuous} for $V=A$, $J=I$ and $p=s_L$.
This together with the assumption $\tr(s_L\restriction_V/q)<\infty$ allows us to apply \thref{MainThm-supp} for $p=s_L$, ensuring that there exists a representing Radon measure $\nu$ for~$L$ with support contained in $\{\alpha\in K:\alpha\!\restriction_V\text{is }q\text{--continuous}\}$.
\endproof

\begin{cor}\thlabel{cor-dense-mainthm}
Let $(A, \tau)$ be a locally convex topological algebra, $B$ a sub-algebra of $A$ which is dense in $(A, \tau)$ and $I:=\{\langle W\rangle: W \text{ finite\ dimensional\ subspace of }B\}.$ Let $L$ be a normalized $\sum A^2$--positive linear functional on $A$ and $Q$ a quadratic module such that $K=K_Q$.

If there exists a $\tau-$continuous Hilbertian seminorm $q$ on $A$ with $\tr(s_L/q)<\infty$ and $\sp(q)\cap K\neq\emptyset$ and if for each $S\in I$ there exists a representing Radon measure $\nu_S$ for $L\!\restriction_S$ with support contained in $K_{Q\cap S}$, then there exists a representing measure for $L$ with support contained in $\sp(q)\cap K$.
\end{cor}

\proof
The $\tau$--continuity of $q$ and the density of $B$ in $(A, \tau)$ provide that  $B$ is dense in $(A, q)$. Moreover, since $L(a^2)=s_L(a)^2$ for all $a\in A$, \eqref{con-L} holds for $p=s_L$, $C=1$ and $V=A$, so Lemma \ref{lem-suff-conc} ensures that $\{\nu_S:S\in I\}$ is $s_L-$concentrated. Then, by \thref{MainThm-dense-supp}, there exists a representing measure for $L$ with support contained in $\sp(q)\cap K$.
\endproof

\begin{rem}\label{rem-sep} In \thref{theorem-main-s_L} and \thref{cor-dense-mainthm} we could actually replace $A$ with $A/\ker(s_L)$, because it is readily seen from the Cauchy-Schwartz inequality that $L$ vanishes on $\ker(s_L)$. Moreover, we have that $V/\ker(s_L\restriction_V)=V/(\ker(s_L)\cap V)=V/\ker(s_L)$ and, whenever
$\tr(s_L\restriction_V/q)<\infty$ for some Hilbertian norm $q$, the space $V/\ker(s_L)$ endowed with the quotient seminorm induced by $s_L$ (and also denoted by $s_L$ with a slight abuse of notation) is separable. Hence, whenever these techniques work, $(V, s_L\restriction_V)$ is essentially a separable space.
\end{rem}

Let us now exploit \thref{theorem-main-s_L} to obtain more concrete sufficient conditions for the existence of a representing measure for $L$ in presence of a fixed Hilbertian seminorm $q$ on $A$.

\begin{cor}\thlabel{thm-um-generating-subspace}
Let $A$ be an algebra generated by a linear subspace $V\subseteq A$,  $L$ a normalized linear functional on $A$, $Q$ a quadratic module in $A$ and $q$ a Hilbertian seminorm on $A$ with $\{\alpha\in K_Q:\alpha\!\restriction_V\text{ is }q\text{--continuous}\}\neq\emptyset$. If
\begin{enumerate}[label = (\alph*)]
\item $L(Q)\subseteq [0,\infty)$,
\item for each $v\in V$, $\sum\limits_{n=1}^\infty \frac{1}{\sqrt[2n]{L(v^{2n})}}=\infty$,
\item $\mathrm{tr}(s_L\!\restriction_V/q)<\infty$, where $s_L(a):=\sqrt{L(a^2)}$ for all $a\in A$,
\end{enumerate}
then there exists a unique representing Radon measure $\nu$ for $L$ with support contained in $\{\alpha\in K_Q:\alpha\!\restriction_V\text{ is }q\text{--continuous}\}$.
\end{cor}

\begin{proof}
Let $J:=\{\langle W\rangle: W \text{ finite\ dimensional\ subspace of }V\}.$ By \cite[Theorem~3.17-(i)]{InKuKuMi22}, the assumptions (a) and (b) guarantee that for each $S\in J$ there exists a unique representing Radon measure $\nu_S$ for $L\!\restriction_S$ with support contained in $K_{Q\cap S}$.
This together with the assumption (c) allows us to apply \thref{theorem-main-s_L}, ensuring that there exists unique representing Radon measure $\nu$ for $L$ with support contained in $\{\alpha\in K_Q:\alpha\!\restriction_V\text{ is }q\text{--continuous}\}$.\end{proof}

\begin{rem}\thlabel{rem-Haviland}
\thref{thm-um-generating-subspace} still holds if we replace the assumptions (a) and (b) with the assumption (a') $L(Pos(K_Q))\subseteq [0,+\infty)$ and in the proof we use \cite[Theorem~3.14]{InKuKuMi22} instead of  \cite[Theorem~3.17]{InKuKuMi22}. However, under this replacement, the uniqueness is not anymore ensured.
\end{rem}

The following lemma provides an explicit construction of a Hilbertian seminorm $q$ as required in \thref{theorem-main-s_L} and so in \thref{thm-um-generating-subspace}.

\begin{lem}
\thlabel{lemma::a}
Let $A$ be an algebra generated by a linear subspace $V\subseteq A$ and $p$ a Hilbertian seminorm on $V$ such that there exists a complete $p-$orthonormal system $\{e_n:n\in\NN\}$ in $V$. Choose $(\lambda_n)_{n\in\NN}\subseteq\RR_{>0}$ such that  $\sum_{n=1}^\infty \lambda_n^2 < \infty$.
Then
\[
q(v):=\sqrt{\sum\limits_{n=1}^\infty\lambda_n^{-2}\langle v, e_n\rangle_p^2}, \quad \forall v\in V
\]
 defines a Hilbertian seminorm on $U:=\left\{ v \in V \, : \, \sum\limits_{n=1}^\infty\lambda_n^{-2}\langle v, e_n\rangle_p^2 < \infty\right\}$ such that $\mathrm{tr}(p\restriction_U/q)<\infty$ and $U$ is dense in $(V, p)$.
\end{lem}

\begin{proof}
For each $a,b\in U$, let us define
$
\langle a,b\rangle_q:=\sum\limits_{n=1}^\infty\lambda_n^{-2}\langle a, e_n\rangle_p\langle b, e_n\rangle_p
$ (note that the Cauchy-Schwartz inequality provides that $\langle a,b\rangle_q<\infty$, since $\langle v,v\rangle_q<\infty$ for all $v\in U$ by the definition of $U$). Then $\langle \cdot,\cdot\rangle_q$ is a symmetric positive semidefinite bilinear form on $U\times U$ and thus, $q(v)=\sqrt{\langle v,v\rangle_q}$ for all $v\in U$ defines a Hilbertian seminorm on $U$.

As $\{e_n:n\in\NN\}$ is a complete $p-$orthonormal system in $V$, we have that Parseval's equality holds and so
$$\forall\ v\in U, \quad p(v)^2=\sum_{n=1}^\infty\langle v, e_n\rangle_{p}^2=\sum_{n=1}^\infty \lambda_n^{2}\lambda_n^{-2}\langle v, e_n\rangle_p^2\leq\left(\sup_{n\in\NN} \lambda_n^2 \right) q(v)^2,$$
i.e. $\forall v \in U$, $p(v)\leq C q(v)$ where $C:=\sup_{n\in\NN} \lambda_n^2$ is finite because of the assumption $\sum_{n=1}^\infty \lambda_n^2 < \infty$.

% by \thref{prop::trace-HS}-(2), the canonical map $u: U/\ker{q}\to U/\ker{s_L}$ uniquely and continuously extends to the completions $\overline{u}: (\overline{U/\ker{q}}, \overline{q})\to (\overline{U/\ker{s_L}}, \overline{s_L})$.
%Hence, by \thref{prop::trace-HS}-(3), to get the conclusion it is enough to show that $\mathrm{tr}(\overline{u}^*\overline{u})<\infty$.

Moreover, since for all $i,j\in\NN$ we have
$$
\langle \lambda_i e_i,\lambda_j e_j\rangle_q=\sum_{n=1}^\infty\lambda_n^{-2}\langle\lambda_i e_i, e_n\rangle_{p}\langle\lambda_j e_j, e_n\rangle_{p}= \sum_{n=1}^\infty\lambda_n^{-2}(\lambda_i\delta_{i,n})(\lambda_j\delta_{j,n})=\delta_{i,j},
$$
the set $\{\lambda_n e_n:n\in\NN\}$ is $q-$orthonormal. Moreover, for all $v\in U$ and all $n\in\NN$ we have that $\langle v,\lambda_n e_n\rangle_q = \lambda_n^{-1}\langle v, e_n\rangle_p$ and so
$$
\forall \ v\in U,\quad \langle v,v\rangle_q^2=\sum_{n=1}^\infty\lambda_n^{-2}\langle v, e_n\rangle_p^2=\sum_{n=1}^\infty\langle v,\lambda_n e_n\rangle_q^2,
$$
i.e.\ Parseval's equality is satisfied, which is equivalent to say that $\{\lambda_n e_n:n\in\NN\}$ is a complete $q-$orthonornal system in $U$ by  \cite[Chapter~5, \S2.3, Proposition~5]{BouTVS}.
%Thus, by \cite[Chapter~5, \S4.6, Lemma~2]{BouTVS}, we have that
%$\mathrm{tr}(\overline{u}^*\overline{u})=\sum_{n=1}^\infty s_L(\lambda_n e_n)^2$.
%Now
Hence, using \eqref{char-trace}, we get that
\begin{equation}\label{finite-square-pL}
\mathrm{tr}(p\restriction_U/q)=\sum_{n=1}^\infty p(\lambda_n e_n)^2=\sum_{n=1}^\infty \lambda_n^2p(e_n)^2=\sum_{n=1}^\infty \lambda_n^2<\infty
\end{equation}
%and so $\mathrm{tr}(\overline{u}^*\overline{u})<\infty$, which is equivalent to $\mathrm{tr}(s_L\restriction_U/q)<\infty$ by \thref{prop::trace-HS}-(3).
As $U$ contains $\{e_n\, : \, n \in\mathbb{N} \}$, we get that $U$ is dense in $V$.
\end{proof}

\begin{rem}\thlabel{rem-orth}\
\begin{enumerate}[label = (\roman*)]
\item If $(V, p)$ is Hausdorff and contains a countable total subset, then the existence of a complete $p-$orthonormal system in $V$ is guaranteed by \cite[Chapter~5, \S2.4, Corollary p. V.24]{BouTVS}. In particular, such a system exists when $(V, p)$ is Hausdorff and separable.
%\item If $\{v_i:i\in I\}$ is a (algebraic) vector space basis of $V$, then it suffices to show that $\sum_{n=1}^\infty\lambda_n^{-2}\langle v_i, e_n\rangle_p^2<\infty$ for all $i\in I$.
%\Tobias{I do not understand what is the purpose of this and what does it mean}
\item If $\{f_n:n\in\NN\}$ is another complete $p-$orthonormal system in $V$, then we get that $\sum_{n=1}^\infty\lambda_n^{-2}\langle v, f_n\rangle_p^2<\infty$ for all $v\in TU$ where $T$ is the linear operator $T\colon V\to V$ given by $e_n\mapsto f_n$ for all $n\in\NN$. Indeed, since $T$ maps a complete $p-$orthonormal system to another complete $p-$orthonormal system, $T$ is orthogonal and so for all $v\in U$
we get that $
\sum_{n=1}^\infty\lambda_n^{-2}\langle Tv, f_n\rangle_p^2=\sum_{n=1}^\infty\lambda_n^{-2}\langle Tv,Te_n\rangle_{p}^2=\sum_{n=1}^\infty\lambda_n^{-2}\langle v,e_n\rangle_{p}^2<\infty.
$
\item Iterating the construction in \thref{lemma::a}, we can show that there exists dense subset $U$ of $(V,p)$ such that $U$ can be equipped with a nuclear topology stronger than the one inherited from $p$.
\end{enumerate}
\end{rem}

Combining \thref{lemma::a} with \thref{theorem-main-s_L}, we obtain the following corollary.
\begin{cor}\thlabel{cor-appl-lemma}
Let $A$ be an algebra generated by a linear subspace $V\subseteq A$ and $L$ a normalized $\sum A^2$--positive linear functional on~$A$ such that there exists a complete $s_L-$orthonormal system $\{e_n:n\in\NN\}$ in~$V$. Choose $U$ and $q$ as in \thref{lemma::a}, and set $J(U):=\{\langle W\rangle: W \text{ finite\ dim.\ subspace of }U\}$.

If $\{\alpha\in X(A):\alpha\!\restriction_U\text{is }s_L\text{--continuous}\}\neq\emptyset$ and for each $S\in J(U)$ there exists a representing Radon measure $\nu_S$ for $L\!\restriction_S$, then there exists a representing measure for $L\restriction_{\langle U\rangle}$ with support contained in $\{\alpha\in X(\langle U\rangle):\alpha\!\restriction_U\text{is }q\text{--continuous}\}$.
\end{cor}

\proof
The assumptions ensure that we can apply \thref{lemma::a} to $(V, s_L\restriction_V)$, which provides a Hilbertian seminorm $q$ on a dense subset $U$ of $(V, s_L\restriction_V)$ such that $\mathrm{tr}(s_L\restriction_U/q)<\infty$. Then $\{\alpha\in X(A):\alpha\!\restriction_U\text{is }s_L\text{--continuous}\}\subseteq\{\alpha\in X(\langle U\rangle):\alpha\!\restriction_U\text{is }q\text{--continuous}\}$ and so $\{\alpha\in X(\langle U\rangle):\alpha\!\restriction_U\text{is }q\text{--continuous}\}\neq\emptyset$. Hence, we can apply \thref{theorem-main-s_L} to $L\restriction_{\langle U\rangle}$ and get the conclusion.
\endproof

In the above corollary the Hilbertian seminorm $q$ on $V$ is not pre-given as in \thref{MainThm} (resp. \thref{MainThm-supp}), but explicitly constructed through \thref{lemma::a}. The price to pay for this is that we obtain an integral representation for the starting linear functional $L$ not on the whole $A$ but just on the sublagebra $\langle U \rangle$ of $A$. Note that the latter subalgebra is actually dense in $(A, s_L)$ (or more in general in $(A,p)$ when $p$ is defined on the whole of $A$, see \thref{lem:tobbi} in the Appendix) and so \thref{cor-appl-lemma} provides a representing measure for $L$ restricted to a dense subalgebra of $(A, s_L)$. However, the representing measure is supported on characters whose restrictions to $U$ lie in the topological dual of $(U,q)$ and the density of $\langle U \rangle$ in $(A, s_L)$ does not allow us to show that is supported on characters whose restrictions to $V$ are in the topological dual of  $(V, s_L\restriction_V)$ and so to get an integral representation for $L$ on the full $A$. The latter effect is not an artefact of the techniques used here. Indeed, if $V=\ell_2$ is endowed with the usual norm $\|\cdot\|_{\ell_2}$ which makes it a Hilbert space, then the associated Gaussian measure (which is the product of infinitely many one-dimensional standard Gaussian measures) gives rise to a functional $L$ on $S(V)$ and the Gaussian measure is the only measure representing $L$. As $s_L\restriction_V=\|\cdot\|_{\ell_2}$, the Gaussian measure cannot be supported on $\{\alpha\in X(S(V)): \alpha\restriction_V \text{ is $s_L\restriction_V-$continuous}\}$ because it is well-known that this set has measure zero (see e.g. \cite[Theorem 1.3]{BeKo88}.

In the case when $U=V$, \thref{cor-appl-lemma} provides a representing measure for the starting $L$ on the whole $A$. This is for example the case when $(V, \tau_V)$ is separable nuclear and $s_L\restriction_V$ is $\tau_V-$continuous, as analyzed in more details in the next subsection.

\subsection{Results on \ref{GenKMP} for $\tau_V$ nuclear}\label{sec:nucl}\

\thref{theorem-main-s_L} and \thref{thm-um-generating-subspace} nicely applies to the case when the generating subspace of the algebra is endowed with a nuclear topology.

\begin{cor}\thlabel{thm-nucelar-criterium::cor}
Let $(V,\tau_V)$ be a nuclear space, $A$ an algebra generated by $V$, $J:=\{\langle W\rangle: W \text{ finite\ dimensional\ subspace of }V\}$ and $K\subseteq X(A)$ closed. Let $L$ be a normalized $\sum A^2$--positive linear functional $L$ on $A$ and $Q$ a quadratic module such that $K=K_Q$.

If for each $S\in J$ there exists a representing Radon measure $\nu_S$ for $L\!\restriction_S$ with support contained in $K_{Q\cap S}$ and $s_L\!\restriction_V$ is $\tau_V$--continuous, then for each Hilbertian seminorm $q$ on $V$ s.t. $\{\alpha\in K_Q:\alpha\!\restriction_V\text{ is }q\text{--continuous}\}\neq\emptyset$ and $\mathrm{tr}(s_L\!\restriction_V/q)<\infty$, there exists a representing measure for $L$ with support contained in $\{\alpha\in K:\alpha\!\restriction_V\text{is }q\text{--continuous}\}$.
\end{cor}
\begin{proof}
As $(V,\tau_V)$ is nuclear and $s_L\!\restriction_V$ is $\tau_V$--continuous, using \thref{prop::trace} and the directnedness of the generating family for $\tau$, we can easily derive that there exists a $\tau_V$--continuous Hilbertian seminorm $q$ on $V$ such that $\mathrm{tr}(s_L\!\restriction_V/q)<\infty$.
Thus, we can apply \thref{theorem-main-s_L} and get the desired conclusion.
\end{proof}

Using exactly the same argument but replacing \thref{theorem-main-s_L} with \thref{thm-um-generating-subspace}, we obtain the following.
\begin{cor}\thlabel{thm-um-generating-subspace::cor}
Let $(V,\tau_V)$ be a nuclear space, $A$ be generated by $V$, $L$ a normalized linear functional on $A$ such that $L(\sum A^2)\subseteq [0,\infty)$ and $Q$ a quadratic module in $A$.
If \begin{enumerate}[label = (\alph*)]
\item $L(Q)\subseteq [0,\infty)$,
\item for each $v\in V$, $\sum\limits_{n=1}^\infty \frac{1}{\sqrt[2n]{L(v^{2n})}}=\infty$,
\item $s_L\!\restriction_V$ is $\tau_V$--continuous,
\end{enumerate}
then, for each Hilbertian seminorm $q$ on $V$ s.t. $\{\alpha\in K_Q:\alpha\!\restriction_V\text{ is }q\text{--continuous}\}\neq\emptyset$ and $\mathrm{tr}(s_L\!\restriction_V/q)<\infty$, there exists a unique representing Radon measure $\nu$ for $L$ such that $\nu(\{\alpha\in K_Q:\alpha\!\restriction_V\text{ is }q\text{--continuous}\})=1$.
\end{cor}

%\begin{cor}
%Let $(A, \tau)$ be a locally convex nuclear topological algebra, $B$ a sub-algebra of $A$ which is dense in $(A, \tau)$ and $L$ a normalized linear functional on $A$ such that $s_L$ is $\tau-$continuous. Denote by $I:=\{\langle W\rangle: W \text{ finite\ dimensional\ subspace of }B\}.$
%
%If for each $S\in I$ there exists a $X(S)-$representing Radon measure $\nu_S$ for $L\!\restriction_S$, then for any Hilbertian seminorm $q$ on $A$ with $\tr(s_L/q)<\infty$ and $\sp(q)\neq\emptyset$ there exists a $\sp(q)-$representing measure for $L$.
%
%\end{cor}
%
%\proof
%As $(A,\tau)$ is nuclear and $s_L$ is $\tau-$continuous, using \thref{prop::trace} and the directnedness of the generating family for $\tau$, we can easily derive that there exists a $\tau$--continuous Hilbertian seminorm $q$ on $A$ such that $\mathrm{tr}(s_L/q)<\infty$. Thus, $L$ is $q-$continuous and $B$ is dense in $(A, q)$. Moreover, if for each $\varepsilon>0$ we set $\delta:=\sqrt{\varepsilon}$, then for each $S\in I$ and any $ a\in B_\delta(s_L)\cap S
%$ we have
%$$ \nu_S(\{\alpha\in X(S):\left|\alpha(a)\right|\geq 1\})\leq \int_{X(S)} \hat{a}^2 d\mu_S=L(a^2)=s_L(a)^2\leq\varepsilon$$
%i.e. $\{\nu_S:S\in I\}$ fulfills \eqref{eq::continuous} for $V=A$, $J=I$ and $p=s_L$.
%Hence, by \thref{MainThm-supp_2}, there exists a $\sp(q)-$representing measure for $L$.
%\endproof
%

We can retrieve \cite[Theorem 13]{Smu17} from \thref{thm-um-generating-subspace::cor} applied to $Q=\sum A^2$. Indeed, in \cite[Theorem 13]{Smu17} the assumption (ii) exactly correspond to (a) and (b) of \thref{thm-um-generating-subspace::cor} for $Q=\sum A^2$ (the alternative assumption (i) corresponds to (a') in Remark \ref{rem-Haviland}), and the assumption of the existence of a $\tau-$continuous seminorm $q$ on $V$ such that  $L(v^2)\leq q(v)^2$ for all $v\in V$ guarantees that $s_L\!\restriction_V(a)\leq q(v)$ for all $v\in V$, i.e. also (c) in \thref{thm-um-generating-subspace::cor} is satisfied.

Note that if there exists a $\tau-$continuous Hilbertian seminorm $q$ on $A$ such that  $L(a^2)\leq q(a)^2$ for all $a\in A$, then not only $s_L$ is $\tau-$continuous but also $L$ itself is $\tau$--continuous, since by the Cauchy-Schwarz inequality we have that
$$
\left|L(a)\right|^2=\left|L(1\cdot a)\right|^2\leq L(1)L(a^2)\leq q(a)^2\qquad\text{for all }a\in A.
$$
% $L$ is $s_L$--continuous.
Viceversa, the continuity of $L$ on certain classes of nuclear topological algebra provides the continuity of $s_L$, allowing us to establish the following result.

\begin{cor}\thlabel{thm-um-generating-subspace::cor2}
Let $(A, \tau)$ be a locally convex nuclear topological algebra which is also barrelled (respectively, has also jointly continuous multiplication), $L$ a $\tau-$continuous linear functional on $A$ and $Q$ a quadratic module in $A$.
If \begin{enumerate}[label = (\alph*)]
\item $L(Q)\subseteq [0,\infty)$,
\item for each $v\in V$, $\sum\limits_{n=1}^\infty \frac{1}{\sqrt[2n]{L(v^{2n})}}=\infty$,
\end{enumerate}
then, for each Hilbertian seminorm $q$ on $V$ s.t. $K_Q\cap\mathfrak{sp}(q)\neq\emptyset$ and $\mathrm{tr}(s_L/q)<\infty$, there exists a unique $(K_Q\cap\mathfrak{sp}(q))$--representing Radon measure $\nu$ for $L$.
\end{cor}
\proof Let us first observe that the $\tau$--continuity of $s_L$ is ensured both when $(A,\tau)$ is barrelled and when has jointly continuous multiplication. Indeed, in the first case the $\tau$--continuity of $L$ ensures the existence of a Hilbertian seminorm $q$ on $A$ such that  $L(a^2)\leq q(a)^2$ for all $a\in A$(see, e.g.\ \cite[Lemma~14]{Smu17}) and so, as observed above, $s_L$ is $\tau$--continuous. In the second case, the joint continuity of the multiplication provides the existence of a $\tau$--continuous seminorm $p$ on $A$ such that $L(ab)\leq p(a)p(b)$ for all $a,b\in A$ and so $s_L(a)^2=L(a^2)\leq p(a)^2$ for all $a\in A$, which shows that $s_L$ is $\tau$--continuous.

Hence, in both cases we can apply \thref{thm-um-generating-subspace::cor} for $V=A$ and get the desired conclusion.
\endproof
%Further, if $q$ is a Hilbertian seminorm such that $\mathrm{tr}(p/q)\leq 1$, then $\mathrm{tr}(s_L/q)\leq 1$ follows immediately from the definition as $s_L\leq p$
 We can easily retrieve \cite[Theorem 15]{Smu17} from \thref{thm-um-generating-subspace::cor2} applied to $Q=\sum A^2$.\\

\section{The case of the symmetric algebra of a nuclear space}\label{Sec:3}

Let us apply the results of Section 2 to the case when $A$ is the symmetric algebra $S(V)$ with $(V,\tau_V)$ nuclear. \thref{thm-um-generating-subspace::cor} immediately gives the following result.

\begin{cor}\thlabel{cor-SV-first}
Let $(V,\tau_V)$ be a nuclear space, $L$ a normalized linear functional on $S(V)$ and $Q$ a quadratic module in $S(V)$. If
\begin{enumerate}[label = (\alph*)]
\item $L(Q)\subseteq [0,\infty)$,
\item for each $v\in V$, $\sum\limits_{n=1}^\infty \frac{1}{\sqrt[2n]{L(v^{2n})}}=\infty$,
\item $s_L\!\restriction_V$ is $\tau$--continuous
\end{enumerate}
then, for each Hilbertian seminorm $q$ on $V$ such that $\mathrm{tr}(s_L\!\restriction_V/q)<\infty$ and $\{\alpha\in K_Q:\alpha\!\restriction_V\text{ is }q\text{--continuous}\}\neq\emptyset$, there exists a unique representing Radon measure $\nu$ for $L$ such that $\nu(\{\alpha\in K_Q:\alpha\!\restriction_V\text{ is }q\text{--continuous}\})=1$.
\end{cor}

We can retrieve \cite[Theorem 16]{Smu17} from \thref{cor-SV-first} applied to $Q=\sum A^2$. Indeed, the definition of nuclear space  in \cite[p.~445]{Smu17} is covered by \thref{def::nuclear-space} (for more details see \thref{schm-nucl}), \cite[Theorem 16]{Smu17} the assumption (ii) exactly correspond to (a) and (b) of \thref{cor-SV-first} for $Q=\sum A^2$ (the alternative assumption (i) corresponds to (a') in Remark \ref{rem-Haviland}), and the assumption of the existence of a $\tau-$continuous seminorm $q$ on $V$ such that  $L(v^2)\leq q(v)^2$ for all $v\in V$ guarantees that $s_L\!\restriction_V(a)\leq q(v)$ for all $v\in V$, i.e. also (c) in \thref{cor-SV-first} is satisfied.

If to each $v \in V$ we associate the operator $A_v(w)=vw$ for any $w\in S(V)$, then we can also retrieve \cite[Theorem 4.3, (i) \text{$\leftrightarrow$} (iii)]{BY75} for such operators from the version of \thref{cor-SV-first} with (a) and (b) replaced by (a') in \thref{rem-Haviland} by taking $L=T$ and $K=\overline{\mathcal{Z}}$ (see also \cite[Theorem 3.11]{IK-probl}).\\

\thref{cor-SV-first} also allows to easily prove the following result.
\begin{cor}\thlabel{simil-Berezansky}
Let $(V,\tau_V)$ be a nuclear space with $\tau_V$ induce by a directed family of seminorms $\P$ on $V$, $L$ a normalized linear functional on $S(V)$ and $Q$ a quadratic module in $S(V)$. If
\begin{enumerate}[label = (\alph*)]
\item $L(Q)\subseteq [0,\infty)$,
\item for each $v\in V$, $\sum\limits_{n=1}^\infty \frac{1}{\sqrt[2n]{L(v^{2n})}}=\infty$,
\item for each $d\in\NN$, there exists $p\in\P$ such that the restriction $L\colon S(V)_d\to\RR$ is $\overline{p}_d$--continuous, where $\overline{p}_d$ is the quotient seminorm on the $d-$th homogeneous component $S(V)_d$ of $S(V)$ induced by the projective tensor seminorm $p^{\otimes d}$,
\end{enumerate}
then, for each Hilbertian seminorm $q$ on $V$ such that $\mathrm{tr}(s_L\!\restriction_V/q)<\infty$ and $\{\alpha\in K_Q:\alpha\!\restriction_V\text{ is }q\text{--continuous}\}\neq\emptyset$, there exists a unique representing Radon measure $\nu$ for $L$ such that $\nu(\{\alpha\in K_Q:\alpha\!\restriction_V\text{ is }q\text{--continuous}\})=1$.
\end{cor}

%For more details on the quotient seminorm $\overline{p}_d$, see \cite{GhInKuMa18}.
 \begin{proof}
Since $L\!\restriction_{S(V)_2}$ is $\overline{p}_2$--continuous for some $p\in\P$, there exists $C>0$ such that $L(v^2)\leq C\overline{p}_2(v)$ for all $v\in V$. Moreover, as $\overline{p}_d$ comes from the projective tensor seminorm $p^{\otimes d}$, we easily get that $\overline{p}_d(v^d)\leq p(v)^d$ holds for all $v\in V$ and all $d\in\NN$, see e.g. \cite[Lemma 3.1]{GhInKuMa18}. Using the latter for $d=2$, we obtain that $L(v^2)\leq C\overline{p}_2(v)\leq Cp(v)^2$ for all $v\in V$, namely that the Hilbertian seminorm $s_L\!\restriction_V$ is $p$--continuous and so $\tau_V$--continuous. Hence, the conclusion follows at once from \thref{cor-SV-first}.
\end{proof}

Using \thref{MainThm-supp_2} instead of \thref{thm-um-generating-subspace::cor}, we can prove a slightly generalization of the classical solution to \ref{GenKMP} for $(V,\tau_V)$ nuclear in \cite[Chapter~5, Theorem 2.1]{BeKo88} (cf. \cite[Chapter 5, Section 2.3]{BeKo88} and \cite{BeSi71}).
%in the reformulation given in \cite[Theorem 6.1]{GhInKuMa18}.

\begin{thm} \thlabel{BKS-thm} Let $(V,\tau_V)$ be a Hausdorff separable nuclear space with $\tau_V$ induced by a directed family $\P$ of Hilbertian seminorms on $V$, $L$ a normalized linear functional on $S(V)$ and $K$ a closed subset of $V^\ast$. For any $n\in\NN$ and $s\in\P$, let $\tilde{s}^{(n)}$ the Hilbertian seminorm on $S(V)_{n}$~given~by $\widetilde{s}^{(n)}(b)\!:=\!\sqrt{\sum\limits_{i=1}^N\sum\limits_{j=1}^N \langle b_{i1}, b_{j1}\rangle_{s}\cdots \langle b_{in}, b_{jn}\rangle_{s}}$ for any $b\!:=\sum\limits_{i=1}^N b_{i1}\!\cdots\!b_{in}\!\in~S(V)_{n}$ with $N \in \mathbb{N}$ and $b_{ik} \in V$ for $k=1, \ldots , n$. If
\begin{enumerate}
\item $L(Q) \subseteq [0,\infty)$, where $Q$ is a quadratic module of $S(V)$ such that $K=K_Q$

\item {there exists a countable subset $E$ of $V$ whose linear span is dense in $(V,\tau_V)$ such that $\sum_{k=1}^\infty\frac{1}{\sqrt[2k]{L(v^{2k})}}=\infty$ for all $v\in span(E)$}
%\item there exists a countable subset $E$ of $V$ whose linear span is dense in $(V,\tau)$ such that, if
%$$ m_0 := \sqrt{L(1)}, \text{ and } m_k := \sqrt{\sup_{f_1, \dots ,f_{2k} \in E} |L(f_1\dots f_{2k})|}, \text{ for } k \ge 1,$$ then the class $C\{ m_k\}$
%is quasi-analytic.
\item\label{itemBK3} For any $d\in\NN$, there exists $p_{{2}d}\in\P$ such that the restriction $L\colon S(V)_{{2}d}\to\RR$ is $\widetilde{p_{{2}d}}^{({2}d)}$--continuous
\item $K\cap V'_{q_2}\neq\emptyset$, where $q_2\in\P$ is such that $\tr(p_2/q_2)<\infty$,
\end{enumerate}
then there exists a representing Radon measure $\mu$ for $L$ with support contained in $K\cap V'_{q_2}$.
\end{thm}

%\Maria{What is the relation between (3) and (c)?}
%\\

\begin{rem}\
%\begin{enumerate}[label = (\roman*)]
%\item

If for each $d$ the map $V\rightarrow \mathbb{R}, v \mapsto L(v^d)$ is $\tau_V-$continuous, then the assumption~(3) in \thref{BKS-thm} holds. Indeed, the $\tau_V-$continuity of the map $V\rightarrow \mathbb{R}, v \mapsto L(v^d)$ implies that for any $d$ there exists a $r_d \in \mathcal{P}$ such that $|L(v^d)| \leq 1$ for all $v \in V$ with $r_d(v) \leq 1$. Then $\left|L\left(\left(\frac{v}{{r}_d(v)}\right)^d\right)\right| \leq 1$ for all $v\in V$ and so
\begin{equation*}
\left|L(v^d)\right| \leq {r}_d(v)^d, \forall v \in V.
\end{equation*}
By using the multivariate polarization identity, this in turn provides that
\begin{equation*}
\left|L(v_1\cdots v_d)\right| \leq \frac{d^d}{d!}{r}_d(v_1)\cdots {r}_d(v_d), \forall v_1,\ldots, v_d \in V.
\end{equation*}
Then, since $(V, \tau)$ is nuclear, \thref{lemma-aux}-(\ref{itemBYNorm1}) below ensures that for any $p_d \in \P$ with $\mathrm{tr}(r_d/p_d) < \infty$ we have
\begin{equation*}
\left|L(a)\right| \leq \frac{\left( \mathrm{tr}(r_d/p_d) d \right)^d}{d!} \widetilde{p}_d^{(d)}(a) , \forall a \in S(V)_d
\end{equation*}
and hence, in particular, $L\restriction_{S(V)_{2d}}$ is $\widetilde{p}_{2d}^{(2d)}-$continuous.

%\item In the proof of \thref{BKS-thm} we also showed that the support of the representing measure is actually contained in $V_{q_2}^\prime$ where $q_2\in\P$ is constructed out of the $\widetilde{p_2}^{(2)}$--continuity of $L\!\restriction_{S(V)_2}$
%(cf. \cite[Remark 1, pp. 72--73]{BeKo88}).
%\end{enumerate}
\end{rem}

\begin{lem}\thlabel{lemma-aux} Let $(V,\tau_V)$ be a separable nuclear space with $\tau_V$ induced by a directed family of Hilbertian seminorms $\P$ on $V$ and $L$ a normalized linear functional.
\begin{enumerate}
\item\label{itemBYNorm1} If for some $d \in \mathbb{N}$, there exists $r\in \P$ and $\widetilde{C}_{L,d}$, such that
$$
| L(v_1\ldots v_d)| \leq \widetilde{C}_{L,d}\, r(v_1) \ldots r(v_d)  \qquad \forall v_1, \ldots , v_d \in V,
$$
then for any $s \in \P$ with $\mathrm{tr}(r/s)<\infty$ we have that
$$
| L(a)| \leq \widetilde{C}_{L,d} \left( \mathrm{tr}(r/s) \right)^d \,  \widetilde{s}^{(d)} (a) \qquad \forall a \in S(V)_d .
$$
\item\label{itemBYNorm2} Let $\ell\in V^\ast$ for some $r\in\P$ and $\alpha_\ell$ the character on $S(V)$ associated to $\ell$, which is uniquely determined by defining $\alpha_\ell(v_1 \ldots v_d):= \ell(v_1) \ldots \ell(v_d)$ for all $d \in \mathbb{N}$ and $v_1 , \ldots , v_d \in V$. If $\ell\in V'_r$ for some $r\in\P$, then the associate character $\alpha_\ell$ on $S(V)$ is such that for any $s \in \P$ with $\mathrm{tr}(r/s)<\infty$ and any $d\in \NN$ the following holds
$$
| \alpha_\ell(a)| \leq \left( r'(\ell) \mathrm{tr}(r/s) \right)^d  \,  \widetilde{s}^{(d)} (a) \qquad \forall a \in S(V)_d .
$$

\item\label{itemBYNorm3}
If the assumption (\ref{itemBK3}) in \thref{BKS-thm} holds with continuity constant $C_{L,2d}$ and $(\lambda_d)_{d\in\NN_0}$ is a sequence of real numbers such that
\begin{equation}\label{seq-lemma}
\sum_{d=0}^\infty \lambda_d^{-2} < \infty,
\end{equation}
then the seminorm defined by
\begin{equation}\label{def-p-tilde}
\tilde{p}(a)^2 := \lambda_0^2 |a^{(0)}|^2 + \sum_{d=1}^\infty \lambda_d^2 C_{L,2d} \left(\widetilde{p_{2d}}^{(d)}(a^{(d)})\right)^2, \quad \forall \ a:=\sum_{d=0}^\infty a^{(d)}\in S(V)
\end{equation}
is Hilbertian and
$$
|L(a)|^2 \leq L(a^2)\leq\left({\sum\limits_{d=0}^\infty \lambda_d^{-2} } \right)\tilde{p}(a)^2 \qquad \mbox{for all } a\in S(V).
$$
\item\label{itemBYNorm4}
Let $C_{L,d}$, $(\lambda_d)_{d\in\NN_0}$ and $\tilde{p}$ as in (\ref{itemBYNorm3}), and for each $d\in\NN$ take a seminorm $q_{2d}\in \P$~such that $\tr(p_{2d}/q_{2d})<\infty$ (such a seminorm always exists by nuclearity). If $(\eta_d)_{d\in\NN_0}$ is a sequence of real numbers such that
\begin{equation}\label{seq-lemma2}
\sum_{d=1}^\infty \frac{\lambda_d^2}{\eta_d^2} C_{L,2d}  \tr(p_{2d}/q_{2d})^d< \infty,
\end{equation}
then the seminorm defined by
\begin{equation}\label{def-q-tilde}
\tilde{q}(a)^2 := \eta_0^2 |a^{(0)}|^2 + \sum_{d=1}^\infty\eta_d^2 \left(\widetilde{q_{2d}}^{(d)}(a^{(d)})\right)^2,  \quad \forall \ a:=\sum_{d=0}^\infty a^{(d)}\in S(V)
\end{equation}
is Hilbertian and such that $\tr(\tilde{p}/\tilde{q}) < \infty$.

\item\label{itemBYNorm5} Let $C_{L,d}$, $(\lambda_d)_{d\in\NN_0}$ and $\tilde{p}$ as in (\ref{itemBYNorm3}), and for each $d\in\NN$ take a seminorm $q_{2d}\in \P$~such that $\tr(p_{2d}/q_{2d})<\infty$ for all $d\in\NN$ and also $\tr(q_{2}/q_{2d})<\infty$ for all $d\in\NN$ with $d\geq 2$ (such a seminorm always exists by nuclearity). If $(\eta_d)_{d\in\NN_0}$ is a sequence of real numbers fulfilling \eqref{seq-lemma2} and
\begin{equation}\label{seq-lemma3}
\sum_{d=1}^\infty  \frac{c^{2d}}{\eta_d^2} < \infty , \forall c>0,
\end{equation}
then $\ell \in V^*$ is $q_2$-continuous if and only if $\alpha_\ell$ is $\tilde{q}$ continuous, i.e. $V'_{q_2}$ ans $\sp(\tilde{q})$ are isomorphic, where $\tilde{q}$ is as in \eqref{def-q-tilde}.
\end{enumerate}
\end{lem}

\proof \

\begin{enumerate}
\item
This is a direct consequence of the multilinear Schwartz kernel theorem for nuclear spaces, see e.g. \cite[Lemma~6.1 and Theorem~6.1]{BeShUs96}.
\item By the $r-$continuity of $\ell$, we obtain that
$
| \alpha_\ell(v_1 \ldots v_d) | \leq r'(\ell)^d r(v_1) \ldots r(v_d)
$ for all $v_1, \ldots , v_d \in V$. Hence, the result directly follows from (\ref{itemBYNorm1}) applied for $L$ replaced with $\alpha_\ell$.

\item

Let $d\in\NN$ and $b:=\sum\limits_{i=1}^N b_{i1}\cdots b_{id}\in S(V)_d$ with $N \in \mathbb{N}$ and $b_{ik} \in V$ for $k=1, \ldots , d$.
Since the assumption (3) of \thref{BKS-thm} holds and $b^2\in S(V)_{2d}$, we have that $L(b^2)\leq C_{2d}\widetilde{p_{2d}}^{(2d)}(b^2)$.
Moreover, since $b^2=\sum\limits_{i=1}^N\sum\limits_{h=1}^N b_{i1}\cdots b_{id}b_{h1}\cdots b_{hd}$, we obtain that

\begin{eqnarray*}
\widetilde{p_{2d}}^{(2d)}(b^2)
%&=&\sqrt{\sum\limits_{i=1}^N\sum\limits_{h=1}^N \sum\limits_{j=1}^N\sum\limits_{k=1}^N \langle b_{i1}, b_{j1}\rangle_{p_{2d}}\cdots \langle b_{id}, b_{jd}\rangle_{p_{2d}}\langle b_{h1}, b_{k1}\rangle_{p_{2d}}\cdots\langle b_{hd}, b_{kd}\rangle_{p_{2d}}    } \\
&=&\sqrt{\sum\limits_{i=1}^N\sum\limits_{h=1}^N \langle b_{i1}, b_{j1}\rangle_{p_{2d}}\cdots \langle b_{id}, b_{jd}\rangle_{p_{2d}} \sum\limits_{j=1}^N\sum\limits_{k=1}^N\langle b_{h1}, b_{k1}\rangle_{p_{2d}}\cdots\langle b_{hd}, b_{kd}\rangle_{p_{2d}}    } \\
&=&\sqrt{\left(\sum\limits_{i=1}^N\sum\limits_{h=1}^N \langle b_{i1}, b_{j1}\rangle_{p_{2d}}\cdots \langle b_{id}, b_{jd}\rangle_{p_{2d}} \right)^2}\\
&=& \sum\limits_{i=1}^N\sum\limits_{h=1}^N \langle b_{i1}, b_{j1}\rangle_{p_{2d}}\cdots \langle b_{id}, b_{jd}\rangle_{p_{2d}}=\widetilde{p_{2d}}^{(d)}(b)^2
\end{eqnarray*}
Hence, we get that
\begin{equation}\label{useful-ineq}
L(b^2)\leq C_{2d}\widetilde{p_{2d}}^{(2d)}(b^2)\leq C_{2d}\widetilde{p_{2d}}^{(d)}(b)^2, \quad\forall b\in S(V)_d.
\end{equation}

Let $(\lambda_d)_{d\in\NN_0}$ as in~\eqref{seq-lemma} and $a:=\sum_{d=0}^\infty a^{(d)}\in S(V)$. Then there exists $D_a\in\NN$ such that $a^{(d)}=0$ for all $d>D_a$ and so we get that
\begin{eqnarray*}
|L(a)|^2\leq |L(a^2)|\!\!\!\!&\leq& \sum_{d=0}^{D_a} \sum_{j=0}^{D_a} | L(a^{(d)})L(a^{(j)})| \leq\sum_{d=0}^{D_a} \sum_{j=0}^{D_a} \sqrt{ L\left(\left(a^{(d)}\right)^2\right)}\sqrt{ L\left(\left(a^{(j)}\right)^2\right)} \\
&=&\left(\sum_{d=0}^{D_a} \sqrt{ L\left(\left(a^{(d)}\right)^2\right)}\right)^2\\
&\leq& \left(\sum_{d=0}^\infty \lambda_d^{-2} \right)\left(\sum_{d=0}^{D_a} \lambda_d^2 L\left(\left(a^{(d)}\right)^2\right)\right)\\
%&\leq&\left(\sum_{d=0}^\infty \lambda_d^{-2} \right) \left(\lambda_0^2 |a^{(0)}|+\sum_{d=1}^n {\lambda_d^2 C_{2d}\ \tilde{p_{2d}}\left(\left(a^{(d)}\right)^2\right)}\right)\\
&\stackrel{\eqref{useful-ineq}}{\leq} &\left(\sum_{d=0}^\infty \lambda_d^{-2} \right) \left(\lambda_0^2 |a^{(0)}|+\sum_{d=1}^{D_a} {\lambda_d^2 C_{2d} \widetilde{p_{2d}}^{(d)}\left(a^{(d)}\right)^2}\right)\\
&=& \left({\sum_{d=0}^\infty \lambda_d^{-2} } \right)\tilde{p}(a)^2.
\end{eqnarray*}

\item
{Since $(V, \tau_V)$ is Hausdorff and separable, we have that for each $d\in\NN$ there exists a complete $q_{2d}-$orthonormal system $E_d$ in $V$ (see \thref{rem-orth}).

Then $\B:=\left\{\frac{1}{\eta_n} e_{i_1}\cdots e_{i_n}  : {n\in \mathbb{N}_0, e_{i_1},\ldots, e_{i_n} \in E_n}\right\}$ is a complete $\tilde{q}-$orthonormal system in $S(V)$} and thus, for any $(\lambda_d)_{d\in\NN_0}$ as in~\eqref{seq-lemma} and $(\eta_d)_{d\in\NN_0}$ as in~\eqref{seq-lemma2}, we obtain that
\begin{eqnarray*}
\tr(\tilde{p}/\tilde{q})&=&\sum_{e\in\B}\tilde{p}(e)^2=\frac{\lambda_0^2}{\eta_0^2}+\sum_{d=1}^\infty \sum_{e_{i_1}, \ldots, e_{i_d} \in E_d } \frac{\lambda_d^2C_{2d}\widetilde{p_{2d}}^{(2d)}(e_{i_1} \cdots e_{i_d})^2}{\eta_d^2} \\
&=&\frac{\lambda_0^2}{\eta_0^2}+ \sum_{d=1}^\infty \sum_{e_{i_1}, \ldots, e_{i_d}\in E_d } \frac{\lambda_d^2C_{2d} p_{2d}(e_{i_1})^2 \cdots p_{2d}(e_{i_d})^2}{\eta_d^2} \\
&=& \frac{\lambda_0^2}{\eta_0^2}+\sum_{d=1}^\infty \frac{\lambda_d^2}{\eta_d^2} C_{2d}\sum_{e_{i_1}\in E_d } p_{2d}(e_{i_1})^2 \cdots \sum_{e_{i_d}\in E_d } p_{2d}(e_{i_d})^2 \\
&=&\frac{\lambda_0^2}{\eta_0^2}+ \sum_{d=1}^\infty \frac{\lambda_d^2}{\eta_d^2} C_{2d}\tr(p_{2d}/q_{2d})^d<\infty.
\end{eqnarray*}
Hence, $\tr(\tilde{p}/\tilde{q}) < \infty$.
\item Let us first show why the existence of a seminorm $q_{2d}$ with the properties as in the statement is guaranteed by the nuclearity of $V$. As $\P$ is directed, for each $d \geq 2$ there exists a seminorm $r_{2d} \in \P$ such that $p_{2d} \leq r_{2d}$ and $q_2 \leq r_{2d}$. Then, by the nuclearity of $V$, we can choose a $q_{2d} \in \P$ such that $\mathrm{tr}(r_{2d}/q_{2d}) <\infty$ and hence, by definition of trace,
$\mathrm{tr}(p_{2d}/q_{2d}) <\infty$ and $\mathrm{tr}(q_{2}/q_{2d}) <\infty$ for all $d \geq 2.$\\
\noindent Let $\ell$ be $q_2$-continuous. Then, for any $d \geq 2$, we get that $$
| \alpha_\ell((a^{(d)})^2) | \stackrel{(\ref{itemBYNorm2})}{\leq} \left( q_2'(\ell) \mathrm{tr}(q_2/q_{2d}) \right)^{2d} \widetilde{q_{2d}}^{(2d)}((a^{(d)})^2), \quad\forall a^{(d)} \in S(V)_d,
$$
while, for $d=1$, we have that $$| \alpha_\ell((a^{(1)})^2) | = \ell(a^{(1)})^2 \leq q_2'(\ell) q_2(a^{(1)})^2, \quad \forall\ a^{(1)} \in S(V)_1 =V.$$

\noindent Moreover, arguing as in (\ref{itemBYNorm3}), it is easy to see that for all $d\in\NN$
$$\widetilde{q_{2d}}^{(2d)}((a^{(d)})^2) = \widetilde{q_{2d}}^{(2d)}(a^{(d)})^2,\quad \forall a^{(d)} \in S(V)_d.$$

\noindent Now, for any $a:=\sum_{d=0}^\infty a^{(d)}\in S(V)$, there exists $D_a\in\NN$ such that $a^{(d)}=0$ for all $d>D_a$. Thus, setting $\tilde{\eta}_d :=  \eta_d   q_2'(\ell)^{-d} \left( 1+  \mathrm{tr}(q_2/q_{2d}) \right)^{-d} $ for all $d\in\NN_0$ and exploiting the previous three inequalities, we get that
\begin{eqnarray*}
| \alpha_\ell(a)|^2 &\leq& | \alpha_\ell(a^2)|
\leq \left(\sum_{d=0}^\infty \tilde{\eta}_d^{-2} \right)\left(\sum_{d=0}^{D_a} \tilde{\eta}_d^2\ \alpha_\ell\!\left(\left(a^{(d)}\right)^2\right)\right)\\
%&\leq&\left(\sum_{d=0}^\infty \lambda_d^{-2} \right) \left(\lambda_0^2 |a^{(0)}|+\sum_{d=1}^n {\lambda_d^2 C_{2d}\ \tilde{p_{2d}}\left(\left(a^{(d)}\right)^2\right)}\right)\\
&\leq &\left(\sum_{d=0}^\infty  \tilde{\eta}_d^{-2} \right) \left( \tilde{\eta}_0^2 |a^{(0)}|+  \tilde{\eta}_1^2 q_2'(\ell)^2 q_2(a^{(1)})^2 + \sum_{d=2}^{D_a} { \tilde{\eta}_d^2 \left( q_2'(\ell) \mathrm{tr}(q_2/q_{2d}) \right)^{2d}\  \widetilde{q_{2d}}^{(d)}\!\left(a^{(d)}\right)^2}\right)\\
&\leq &\left(\sum_{d=0}^\infty  \tilde{\eta}_d^{-2} \right)  \tilde{q}(a)^2,
\end{eqnarray*}
which provides the $\tilde{q}-$continuity of $\alpha_\ell$ since $\left(\sum_{d=0}^\infty  \tilde{\eta}_d^{-2} \right)<\infty$ by \eqref{seq-lemma3}.

\noindent Conversely, if $\alpha_\ell$ is $\tilde{q}-$continuous, then there exists $C\geq0$ such that
$$
| \ell(v)| = |\alpha_\ell( v)| \leq C \tilde{q}(v) = C \eta_1 q_2(v),\quad  \forall v\in V.
$$
\end{enumerate}
\endproof

\begin{proof}[Proof of \thref{BKS-thm}]\ \\
Let $I:=\{\langle W\rangle: W\text{ finite dimensional\ subspace of }span(E)\}$.

Recall that $(X(S(V)),\tau_{X(S(V))})$ is isomorphic to $V^\ast$ equipped with the weak topology. Then, by the generalization of the classical Nussbaum theorem to any finitely generated algebra (see e.g. \cite[Theorem~3.16]{InKuKuMi22}), the assumptions (1) and (2) ensure that for each $S\in I$ there exists a unique $K_{Q\cap S}$--representing measure $\nu_S$ for $L\!\restriction_S$. Moreover, the separability and the nuclearity of $(V, \tau_V)$ as well as the assumptions (1) and (3) ensure that we can apply \thref{lemma-aux} and get two Hilbertian seminorms $\tilde{p}$ and $\tilde{q}$ on $S(V)$ such that $\tr(\tilde{p}/\tilde{q})<\infty$ and $L(a^2)\leq \left({\sum_{d=0}^\infty \lambda_d^{-2} } \right)\tilde{p}(a)^2$ for all $a\in S(V)$. Thus, by Lemma \ref{lem-suff-conc}, $\{\nu_S:S\in I\}$ is \ $\tilde{p}$-concentrated.

Also the density of $span(E)$ in $(V, \tau_V)$ given by assumption (2) implies the density of $S(span(E))$ in $(S(V), \tilde{q})$. Then, by \thref{lemma-aux}-(\ref{itemBYNorm5}) we have that $\sp(\tilde{q})$ is isomorphic to $V_{q_2}^\prime$. Thus, exploiting also the assumption (4), the conclusion follows by applying \thref{MainThm-dense-supp} to $A:=S(V)$, $q=\tilde{q}$, $B:=S(span(E))$, and $p=\tilde{p}$.
\end{proof}

{We can retrieve \cite[Chapter~5, Theorem 2.1]{BeKo88} from \thref{BKS-thm}, because their definition of nuclear space $(V, \tau)$ is covered by \thref{def::nuclear-space} (for more details see \thref{bk-nucl}), their regularity assumption on the starting sequence \cite[Chapter 5, Section 2.1, p.52]{BeKo88} corresponds to \thref{BKS-thm}-(3), their positivity assumption \cite[Chapter 5, (2.1)]{BeKo88} is equivalent \thref{BKS-thm}-(1), and those together with their growth condition in \cite[Chapter 5, (2.5)]{BeKo88} imply that \thref{BKS-thm}-(2) holds. For the convenience of the reader, we restate their growth condition in our setting
\begin{eqnarray}\label{cond-BK}
& \ & \exists\ E\subset V \text{countable} \text{ s.t.\! $span(E)$ is dense in $(V, \tau)$ and $C\{ z_k\}$ is quasi-analytic,}\\
& & \text{where } z_k:=\left(\sup_{v\in E} p_{2k}(v)\right)^k\sqrt{\sup_{v_1, \ldots, v_{2k}\in V}\left(\frac{|L(v_1\cdots v_{2k})|}{\widetilde{p_{2k}}^{(2k)}(v_1\cdots v_{2k})}\right)} \quad \forall k\in\NN,\nonumber\\
& &  \text{$p_{2k}$ and $\widetilde{p_{2k}}^{(2k)}$ are as in \thref{BKS-thm}-(3).\nonumber}
\end{eqnarray}
and prove in detail the above mentioned implication.}

\begin{prop}\thlabel{lemma-Carleman-span}
Let $(V,\tau_V)$ be a separable nuclear space with $\tau_V$ induced by a directed family of seminorms $\P$ on $V$, $L\colon S(V)\to\RR$ linear such that (1) and (3) of \thref{BKS-thm} hold. If \eqref{cond-BK} is fulfilled, then \thref{BKS-thm}-(2) holds.
\end{prop}

\begin{proof}
{Let us preliminarily observe that for each $k\in\NN$ we have
\begin{equation}\label{ineq-definit}
m_k := \sqrt{\sup_{v_1, \dots ,v_{2k} \in E} |L(v_1\dots v_{2k})|}\leq z_k.
\end{equation}}
Moreover, the following properties hold:
\begin{enumerate}[label = (\alph*)]
\item $(m_k)_{k\in\NN}$ is log-convex as $L(ab)^2\leq L(a^2)L(b^2)$ for all $a,b\in S(V)$.
\item for each $a\in S(V)$ the sequence $(\sqrt{L(a^{2k})})_{k\in\NN}$ is increasing (by repeated applications of the Cauchy-Schwartz inequality).
\item for each $a\in S(V)$ and $k\in\NN$, the map $a\mapsto \sqrt[2^k]{L(a^{2^k})}$ defines a semi-norm on $S(V)$ (by \cite[Remark 3.2-(ii)]{TAMP-cmp} )
\item for each $f\in E$ and $k\in\NN$, $\sqrt{L(f^{2k})}\leq m_k$ (by definition of $m_k$).
\end{enumerate}

Fix $k\in\NN$ and $v=\sum_{i=1}^l\lambda_iv_i$ with $v_i\in E$ and $\lambda_i\in\RR$. Let us choose $d\in\NN$ such that $2^d\leq 2k\leq 2^{d+1}$. Then
\begin{eqnarray}\label{bound-2k}
\sqrt[2k]{L(v^{2k})}&\stackrel{(b)}{\leq}&\sqrt[2^{d+1}]{L(v^{2^{d+1}})}
\stackrel{(c)}{\leq}\sum_{i=1}^l\left|\lambda_i\right|\sqrt[2^{d+1}]{L(v_i^{2^{d+1}})}
\leq\sum_{i=1}^l\left|\lambda_i\right|\sqrt[2^d]{\sqrt{L(v_i^{2\cdot 2^d})}}\nonumber\\
&\stackrel{(d)}{\leq}&\sum_{i=1}^l\left|\lambda_i\right|\sqrt[2^d]{m_{2^d}}
\leq\left(1+\sum_{i=1}^l\left|\lambda_i\right|\right)\sqrt[2^d]{m_{2^d}}
\leq K_v\sqrt[2k]{m_{2k}}.
\end{eqnarray}
where in the last inequality we used that $(\sqrt[k]{m_k})_{k\in\NN}$ is increasing, see e.g. \cite[Corollary~4.2]{In16} and $K_v:=1+\sum_{i=1}^l\left|\lambda_i\right|>0$.

Since the class $\C\{{z_{k}}\}$ is quasi-analytic and \eqref{ineq-definit} holds, also $\C\{{m_{k}}\}$. This together with the property (a) ensures that $\C\{\sqrt{m_{2k}}\}$ is quasi-analytic, see e.g. \cite[Lemma 4.3]{In16}. Hence, by the Denjoy-Carleman theorem, $\sum_{k=1}^\infty\frac{1}{\sqrt[2k]{m_{2k}}}=\infty$ holds. This combined with \eqref{bound-2k} provides the conclusion.
\end{proof}

%----------------------------------------------------------

\section{Appendix}\label{sec:appendix}

In the following we first explain the relation between the notion of trace of a Hilbertian seminorm w.r.t. to another and the classical definition of trace of a positive continuous operator on a Hilbert space. Then we compare the definition of nuclear space used in this article due to Yamasaki \cite{Ya85} with the more traditional ones due to Grothendieck \cite{Gr55} and Mityagin \cite{Mi61}, and with the definitions of this concept given by Berezansky and Kondratiev in \cite[p.~14]{BeKo88} and by Schm\"udgen in \cite[p.~445]{Smu17}, whose results we compared to ours in Section \ref{Sec:3}. Finally, we provide a complete proof of the measure theoretical identity \eqref{eq-sigma-algebras}, which we exploited in the proof of \thref{MainThm-supp_2}.

\subsection{Trace of positive continuous operators on Hilbert spaces}\label{sec: app-trace}\ \\
Let us start by recalling the definition of trace of a positive continuous operator on a Hilbert space, which we also denote with the symbol $\mathrm{tr}$.

\begin{dfn}[cf.\ {\cite[V.50, (24')]{BouTVS}}]\thlabel{def::op-trace}
Given a Hilbert space $(H,\langle\cdot, \cdot\rangle)$, the \emph{trace} of a continuous and positive operator $f\colon H\to H$ is defined as
\begin{equation}\label{eq::trace1}
\mathrm{tr}(f):=\sup_{e_1,\ldots,e_n}\sum_{i=1}^n\langle e_i,f(e_i)\rangle,
\end{equation}
where $n$ ranges over $\NN$ and $e_1,\ldots,e_n\in H$ ranges over the set of all finite sequences that are orthonormal w.r.t.\ $\langle\cdot,\cdot\rangle$.
\end{dfn}
In fact, by \cite[V.48, Lemma~2]{BouTVS}, we have that for every complete orthonormal system $\{e_i:i\in\Omega\}$ in $H$ the following holds
\begin{equation}\label{eq::trace2}
\mathrm{tr}(f)=\sum_{i\in\Omega}\langle e_i, f(e_i)\rangle.
\end{equation}
If $D\subseteq H$ is dense, then there exists a complete orthonormal system in $H$ that is contained in $D$. Therefore, in \eqref{eq::trace1} it suffices to let $e_1,\ldots,e_n\in H$ range over the set of all finite sequences in $D$ that are orthonormal w.r.t.\ $\langle\cdot,\cdot\rangle$.\\

For the convenience of the reader, we also recall here some fundamental classes of operators that will be needed in showing the relation between traces mentioned above.
\begin{dfn}\thlabel{def::trace-class}
Given a Hilbert space $(H,\langle\cdot, \cdot\rangle)$, we say that a bounded linear operator $f:H\to H$ is \emph{trace-class} if $tr(\sqrt{f^*f})<\infty$, where $f^\ast$ denotes the adjoint of $f$. The positive bounded operator $\sqrt{f^*f}$ is called \emph{absolute value} of $f$.
\end{dfn}

\begin{dfn}\thlabel{def::HS+nucl}
Given two Hilbert spaces $(H_1,p_1)$ and $(H_2,p_2)$, we say that a continuous operator $f: H_1\to H_2$ is
\begin{enumerate}
\item \emph{Hilbert-Schmidt (or quasi-nuclear)} if $\mathrm{tr}(f^*f)<\infty$.% (c.f. \cite[V.59, Definition~9]{BouTVS}).
\item \emph{nuclear} if there exist $(v_n)_{n\in\NN}\subseteq H_1$ and $(w_n)_{n\in\NN}\subseteq H_2$ such that
$$
\sum_{n=1}^\infty p_1(v_n)p_2(w_n)<\infty\quad\text{and}\quad f(\cdot)=\sum_{n=1}^\infty\langle\cdot, v_n\rangle_{p_1}w_n.
$$
\end{enumerate}
\end{dfn}
Note that $(v_n)_{n\in\NN}\subseteq (H_1,p_1)$ and $(w_n)_{n\in\NN}\subseteq (H_2,p_2)$ can be chosen to be orthogonal (see, e.g., \cite[Corollary, p.~494]{Tr67}).

\begin{prop}\thlabel{restr-nuclear}
Let $f\colon (H_1,p_1)\to (H_2,p_2)$ be a nuclear operator. If $H\subseteq H_1$ closed, then $f\!\restriction_H\colon H\to\overline{f(H)}$ is also nuclear.
\end{prop}
\begin{proof}
Since $f$ is nuclear, there exists $(v_n)_{n\in\NN}\subseteq (H_1,p_1)$ and $(w_n)_{n\in\NN}\subseteq (H_2,p_2)$ orthogonal such that $f(\cdot)=\sum_{n=1}^\infty\langle\cdot, v_n\rangle_{p_1}w_n$ and $\sum_{n=1}^\infty p_1(v_n)p_2(w_n)<\infty$. Then $f(v_n)=\langle v_n, v_n\rangle_{p_1}w_n$ for all $n\in\NN$, since $(v_n)_{n\in\NN}\subseteq (H_1,p_1)$ is orthogonal, and so $(w_n)_{n\in\NN}\subseteq\overline{f(H)}$. Furthermore, for each $n\in\NN$ there exist $x_n\in H, y_n\in H^\perp$ such that $v_n=x_n+y_n$. Thus,
$$
f(x)=\sum_{n=1}^\infty\langle x, x_n+y_n\rangle_{p_1}w_n=\sum_{n=1}^\infty\langle x, x_n\rangle_{p_1}w_n\quad\text{for all }x\in H.
$$
Moreover, $\langle x_n,y_n\rangle_{p_1}=0$ implies that $p_1(x_n)\leq p_1(x_n+y_n)=p_1(v_n)$ for all $n\in\NN$ and hence, $\sum_{n=1}^\infty p_1(x_n)p_2(w_n)\leq\sum_{n=1}^\infty p_1(v_n)p_2(w_n)<\infty$.
\end{proof}

We are ready now to relate \thref{def::trace} and \thref{def::op-trace}.

\begin{rem}\thlabel{rem::hilbertian}
A Hilbertian seminorm $p$ on a real vector space $V$ can be always used to construct a Hilbert space out of~$V$. Indeed, $p$ induces a seminorm on $V_p:=V/\ker(p)$ given by $v+\ker(p)\mapsto p(v)$ and denoted, with a slight abuse of notation, also by $p$. Thus, $(V_p, p)$ is a pre-Hilbert space, as $p$ clearly induces an inner product on $V_p$. Now, $V_p$ is dense in the completion $\overline{V}_p$ of $(V_p, p)$ and so $p$ extends to a norm $\overline{p}$ on $\overline{V}_p$ which makes $(\overline{V}_p,\overline{p})$ a Hilbert space. \par\medskip
\end{rem}

\begin{prop}\thlabel{prop::trace-HS}\ \\
Let $p$ and $q$ be two Hilbertian seminorms on a real vector space $V$.
\begin{enumerate}
\item If $\ker(q)\subseteq\ker(p)$ then $u\colon V_q\to V_p, v+\ker(q)\mapsto v+\ker(p)$ is well-defined. Note that $u$ is injective iff $\ker(q)=\ker(p)$.
\item If there exists $C>0$ such that $p\leq C q$, then $u$ is continuous and uniquely continuously extends to $\overline{u}\colon(\overline{V}_q,\overline{q})\to(\overline{V}_p,\overline{p})$. Moreover, $\overline{u}$ is injective iff for any Cauchy sequence $(v_n)$ in $V_q$ s.t. $u(v_n)$ converges to $0$ in $p$ we have that $v_n$ converges to $0$  in $q$.
\item $\mathrm{tr}(p/q)<\infty$ if and only if $\overline{u}$ is Hilbert-Schmidt, i.e., $\mathrm{tr}(\overline{u}^\ast\overline{u})<\infty$, where $\overline{u}^\ast$ denotes the adjoint of $\overline{u}$.
\end{enumerate}
\end{prop}

\proof\

(1) For any $v\in V$, let us set for convenience $[v]_q:=v+\ker(q)$ and $[v]_p:=v+\ker(p)$. Recalling the notation and the properties introduced in \thref{rem::hilbertian}, it is easy to see that $\ker(q)\subseteq\ker(p)$ implies (1), because under this assuption $[x]_q=[y]_q$ implies $x-y\in\ker(p)$ and so $[x]_p=[y]_p$, i.e. $u([x]_q)=u([y]_q)$.

Moreover, suppose that $u$ is injective and that there exists $v\in\ker(p)\setminus\ker(q)$. Then $[v]_q\neq[0]_q$ and $[v]_p=[0]_p$. Hence, on the one hand the injectivity of $u$ ensures that $u([v]_q)\neq [0]_p$, but on the other hand $u([v]_q)=[v]_p=[0]_p$ which leads to a contradiction. Conversely, if $\ker(p)=\ker(q)$, then $u$ is the identity which is clearly injective.

 (2) Suppose there exists $C>0$ such that $p\leq C q$. Then $\ker(q)\subseteq\ker(p)$ and so $u$ is well-defined by (1). Also, for any $[v]_q\in V_q$ we have $p(u([v]_q))=p([v]_p)=p(v)\leq Cq(v)=Cq([v]_q)$, i.e. $u$ is continuous and so can be uniquely extended to the completions giving the desired $\overline{u}$.

For proving the second part of (2), suppose that $\overline{u}$ is injective and let $(v_n)$ be a Cauchy sequence in $V_q$ s.t. $p(u(v_n))\to 0$. Then, by completeness, there exists $w\in\overline{V}_q$ such that $\overline{q}(v_n-w)\to 0$ and so, by continuity of $\overline{u}$, $\overline{p}(\overline{u}(v_n)- \overline{u}(w))\to 0$. Therefore, $\overline{p}(\overline{u}(w))\leq \overline{p}(\overline{u}(v_n)- \overline{u}(w))+\overline{p}(\overline{u}(v_n))=\overline{p}({u}(v_n)- \overline{u}(w))+{p}({u}(v_n))\to 0$, i.e. $\overline{p}(\overline{u}(w))=0$ that is $\overline{u}(w)=0$. Hence, the injectivity of $\overline{u}$ implies that $w=0$ and so that $q(v_n)=\overline{q}(v_n)=\overline{q}(v_n-w)\to 0$.

Conversely, suppose that for any Cauchy sequence $(v_n)$ in $V_q$ s.t. $p(u(v_n))\to 0$ we have $q(v_n)\to 0$. If $w\in\ker(\overline{u})\subseteq\overline{V}_q$, then there exists a Cauchy sequence $(w_n)$ in $V_q$ converging to $w$, i.e. $\overline{q}(w_n-w)\to 0$. By continuity of $\overline{u}$, we have that $\overline{p}(\overline{u}(w_n)- \overline{u}(w))\to 0$ but $\overline{u}(w)=0$ and so $p(u(w_n))=\overline{p}(\overline{u}(w_n))\to 0$. Hence, our assumption implies $q(w_n)\to 0$ and so $\overline{q}(w)\leq \overline{q}(w_n-w)+q(w_n)\to 0$, which is equivalent to $w=0$ and so provides the injectivity of $\overline{u}$.

 (3) directly follows from the following observation
 \begin{eqnarray*}
\mathrm{tr}(\overline{u}^\ast\overline{u})&\stackrel{\eqref{eq::trace1}}{=}& \sup_{e_1,\ldots,e_n}\sum_{i=1}^n\langle [e_i]_q,\overline{u}^\ast\overline{u}([e_i]_q)\rangle_{\overline{q}}\\
&=& \sup_{e_1,\ldots,e_n}\sum_{i=1}^n\langle\overline{u}([e_i]_q),\overline{u}([e_i]_q)\rangle_{\overline{p}}\\
&=& \sup_{e_1,\ldots,e_n}\sum_{i=1}^n\langle[e_i]_p,[e_i]_p\rangle_{\overline{p}}
= \sup_{e_1,\ldots,e_n}\sum_{i=1}^n\langle e_i,e_i\rangle_p
\stackrel{\text{Def.} \ref{def::trace}}{=}\mathrm{tr}(p/q),
\end{eqnarray*}
where $n$ ranges over $\NN$ and $e_1,\ldots,e_n\in V$ ranges over the set of all finite sequences that are orthonormal w.r.t.\ $\langle\cdot,\cdot\rangle_q$.
\endproof

\begin{prop}\thlabel{prop::separable}
Let $p$ and $q$ be two Hilbertian seminorms on $V$.  If $\ker(p)=\ker(q)$ and $\mathrm{tr}(p/q)<\infty$, then $V_q$ is separable.
\end{prop}
\proof
Suppose that $V_q$ is not separable. Then there exists $(e_j)_j\in J$ orthonormal basis of $V_q$ with $J$ uncountable. Since $q(e_j)=1$ for all $j\in J$ and $\ker(p)=\ker(q)$, we have that $p(e_j)>0$ for all $j\in J$. However, $\mathrm{tr}(p/q)<\infty$ implies that $\sup_{n\in\NN}\sup_{j_1, \ldots, j_n\in J}\sum_{k=1}^n p(e_{j_k})^2<\infty$ and so for all but countably many $n-$tuples in $(e_j)_j\in J$ we have $\sum_{k=1}^n p(e_{j_k})^2=0$, which contradicts the fact that $p(e_j)>0$ for all $j\in J$.
\endproof

\begin{cor}
Let $p$ and $q$ be two Hilbertian seminorms on $V$ s.t. there exists $C>0$ such that $p\leq Cq$ then $\overline{u}$ injective and Hilbert-Schmidt implies that $V_q$ is separable.
\end{cor}

\begin{cor}
Let $A$ be an algebra generated by a linear subspace $V\subseteq A$, and $L$ a normalized linear functional on $A$ such that $L(\sum A^2)\subseteq [0,\infty)$.
If $q$ is a Hilbertian seminorm $q$ on $V$ such that $\tr(s_L\restriction_V/q)<\infty$, then
the space $V/\ker(s_L)$ endowed with the quotient seminorm induced by $s_L$ (and also denoted by $s_L$ with a slight abuse of notation) is separable.
\end{cor}
\proof
Let us endow $V_q$ with the quotient seminorm induced by $q$, which we also denote by $q$ with a slight abuse of notation.  As $\tr(s_L\restriction_V/q)<\infty$, \thref{prop::trace}-(i) ensures the $q-$continuity of $s_L\restriction_V$ and so that $\ker(q) \subset \ker(s_L)$. Then the quotient seminorm induced on $V_q$ by $s_L\restriction_V$ actually reduces to itself, i.e. $\forall v\in V, \ \inf\{s_L(v+w): w\in\ker(q)\}=s_L(v)$, and can be continuously extended to a norm $\overline{p}_L$ on the completion $\overline{V_q}$ of $V_q$. Hence, both $(\overline{V_q}, \overline{p}_L)$ and $(\overline{V_q}, \overline{q})$ are Hilbert spaces.

Consider $\overline{\ker(s_L)}$ in $(\overline{V_q}, \overline{q})$ and denote by $\tilde{q}$ the quotient norm induced on $\overline{V_q}/\overline{\ker(s_L)}$ by $\overline{q}$ (respectively by $\tilde{s_L}$ the quotient norm induced on $\overline{V_q}/\overline{\ker(s_L)}$ by~$\overline{s_L}$). Then $\tilde{q}$ is a Hilbertian norm as $(\overline{V_q}, \overline{q})$ is a Hilbert space and we have the following orthogonal decomposition $\overline{V_q} = \overline{\ker(s_L)} \oplus \overline{\ker(s_L)}^\perp$. Then $(\overline{V_q}/\overline{\ker(s_L)}, \tilde{q})$ is a Hilbert space. Hence, denoting by $\pi$ the orthogonal projection $\overline{V_q}$ onto $\overline{\ker(s_L)}^\perp$, we get $\overline{V_q}/\ker(\pi)\cong\pi(\overline{V_q})$, i.e. $\overline{V_q}/\overline{\ker(s_L)}\cong\overline{\ker(s_L)}^\perp$. Exploiting this isomorphism and the fact that $\ker(s_L)$ is closed in $(\overline{V_q}, \overline{q})$, it is easy to see that any finite $\tilde{q}$-orthonormal subset $\{\tilde{e_j}\}_{j=1, \ldots , J}$ with $J \in \mathbb{N}$ in $\overline{V_q}/\overline{\ker(s_L)}$ provides a finite $\overline{q}-$orthonormal subset $\{h_j\}_{j=1, \ldots , J}:=\pi^{-1}\left(\{\tilde{e_j}\}_{j=1, \ldots , J}\right)$ in $\overline{V_q}$. By density, for each $n\in \mathbb{N}$ and each $j \in \{1, \ldots , J\}$, we can choose $h_j^{(n)} \in V$ such that ${\overline{q}}(h_j - h_j^{(n)}) \leq 1/n$. Orthogonalizing $\{h^{(n)}_j\}_{j \in \{1, \ldots , J\}}$ via the Gram-Schmidt process, we obtain a $q-$orthogonal subset $\{e_j^{(n)}\}_{j \in \{1, \ldots , J\}}$ in $V$ defined inductively by $e^{(n)}_1 := h^{(n)}_1$ and ${e}^{(n)}_k :=h^{(n)}_k - \sum_{j=1}^{k-1} {\frac{\langle h^{(n)}_k , e^{(n)}_j\rangle_q}{\langle e^{(n)}_j , e^{(n)}_j\rangle_q}} e^{(n)}_j$ for all $k\geq 2$. Defining $\tilde{e}^{(n)}_k:=\frac{{e}^{(n)}_k}{q({e}^{(n)}_k)}$ for all $k\in\NN$, we get a $q-$orthonormal subset in V. So for each $k \in \{ 1, \ldots , J\}$ as $n\to\infty$ we get inductively that $\overline{q}({e}^{(n)}_k - h_k)\to 0$  and hence $\overline{s_L}({e}^{(n)}_k - h_k)\to 0$. Thus, for each $\varepsilon >0$ there exists $N\in \NN$ such that  $\overline{s_L}(e^{(N)}_j - h_j) \leq \varepsilon/J$ for all $j$ and so
$$
\sum_{j=1, \ldots , J}\tilde{s_L}(\tilde{e_j})=\sum_{j=1, \ldots , J}\overline{s_L}(h_j)\leq \sum_{j=1, \ldots , J}s_L(e^{(N)}_j) + \varepsilon \leq \tr(s_L\restriction_V/q)+ \varepsilon.
$$
As this holds for an arbitrary finite $\tilde{q}$-orthonormal subset $\{\tilde{e_j}\}_{j=1, \ldots , J}$, we have by the definition of the trace that $\tr(\tilde{s_L}/\tilde{q}) \leq \tr(s_L\restriction_V/q)+ \varepsilon$, which together with $\tr(s_L\restriction_V/q)<\infty$ implies $\tr(\tilde{s_L}/\tilde{q})<\infty$. Then, since the kernel of $\tilde{s_L}$ in $\overline{V_q}/{\ker(s_L)}$ is clearly trivial, so it is the kernel of $\tilde{q}$ in $\overline{V_q}/{\ker(s_L)}$. Hence, \thref{prop::separable} provides that $\left(\overline{V_q}/{\ker(s_L)}\right)/\ker(\tilde{q})$ is separable, i.e. $\overline{{V_q}}/{\ker(s_L)}$ is separable. As the latter space is also metric, we have that its subspace ${V_q}/{\ker(s_L)}$ is also separable. Moreover,  since $\ker(q)\subseteq\ker(s_L)\subseteq V$, we get that ${V_q}/{\ker(s_L)}\cong V /\ker(s_L)$ and so $V /\ker(s_L)$ is also separable.
\endproof

\subsection{Other definitions of nuclear space}\label{sec:app-nuclear}\

The definition of nuclear space used in this article, namely \thref{def::nuclear-space}, is due to Yamasaki but it is equivalent to the more traditional definitions of nuclear space due to Grothendieck \cite{Gr55} and Mityagin \cite{Mi61}, which we report here for the convenience of the reader (see e.g. \cite[Theorems~A.1, A.2]{Ya85} for a proof of these equivalences).

\begin{dfn}%[cf.\ {\cite[Definition~A.2]{Ya85}}]
\thlabel{def::nuclear-space-Grothendieck}
A TVS $(V,\tau)$ is called \emph{nuclear} if $\tau$ is induced by a family $\P$ of seminorms on $V$ such that for each $p\in\P$ there exists $q\in\P$ and there exist sequences $(v_n)_{n\in\NN}\subseteq V,(l_n)_{n\in\NN}\subseteq V^\ast$ with the following property
$$
\sum_{n=1}^\infty p(v_n)q^\prime(l_n)<\infty\text{ for all } v\in V\text{ with }v=\sum_{n=1}^\infty l_n(v)v_n\text{ w.r.t.\ }p,
$$
here $q^\prime$ denotes the dual norm of $q$.
\end{dfn}

\begin{dfn}%[cf.\ {\cite[Definition~A.4]{Ya85}}]
\thlabel{def::nuclear-space-Mityagin}
A TVS $(A,\tau)$ is called \emph{nuclear} if $\tau$ is induced by a family $\P$ of seminorms on $V$ such that for each $p\in\P$ there exists $q\in\P$ with
$
d_n(U_q,U_p)\in\mathcal{O}(n^{-\lambda})
$
for some $\lambda>0$, where $U_p$ denotes the (closed) semiball of $p$ and $d_n(U_q,U_p)$ denotes \emph{$n$--dimensional width of $U_q$ w.r.t.\ $U_p$}, that is $d_n(U_q,U_p)$ is defined as
$$
\inf\{c>0: U_q\subseteq V_{n-1}+cU_p\text{ for some }V_{n-1}\subseteq V\text{ with }\dim(V_{n-1})=n-1\}.
$$
\end{dfn}

Using \thref{prop::trace-HS}, it is easy to establish that \thref{def::nuclear-space} coincides with the following one.
\begin{dfn}\thlabel{def::nuclear-space2}
A TVS $(V,\tau)$ is called \emph{nuclear} if $\tau$ is induced by a directed family $\P$ of Hilbertian seminorms on $V$ such that for each $p\in\P$ there exists $q\in\P$ with $\ker(q)\subseteq \ker(p)$ and the continuous extension $\overline{u}\colon(\overline{V}_q,\overline{q})\to(\overline{V}_p,\overline{p})$ of the canonical map $u: (V_q, q)\to (V_p, p)$ is Hilbert-Schmidt, i.e. $\mathrm{tr}(\overline{u}^*\overline{u})<\infty$.\end{dfn}

This equivalent refomulation of \thref{def::nuclear-space} allows more easily to see its relation with the definitions of this concept given by Berezansky and Kondratiev in \cite[p.~14]{BeKo88} and by Schm\"udgen in \cite[p.~445]{Smu17} whose results we compare to ours in  Section \ref{Sec:3}.

\begin{dfn}[cf.\ {\cite[p.~14]{BeKo88}}]
\thlabel{def::nuclearBeKo}
Let $I$ be a directed index set and $(H_i,p_i)_{i\in I}$ a family of Hilbert spaces such that $V:=\bigcap_{i\in I }H_i$ is dense in each $(H_i,p_i)$ and for all $i,j\in I$ there exists $k\in I$ with $i,j\leq k$ and $(H_k,p_k)\subseteq (H_i,p_i)$  as well as $(H_k,p_k)\subseteq (H_j,p_j)$.  The space $V$ endowed with the topology $\tau$ induced by $\P:=\{p_i:i\in I\}$ is called nuclear if for each $i\in I$ there exists $j\geq i$ in $I$ such that the embedding $(H_j,p_j)\subseteq (H_i,p_i)$ is Hilbert-Schmidt.
\end{dfn}

\begin{rem}\thlabel{bk-nucl}
Berezansky and Kondratiev's \thref{def::nuclearBeKo} of nuclear space is covered by \thref{def::nuclear-space2} (and thus, by \thref{def::nuclear-space}). Indeed, let $(V,\tau)$ be a nuclear space with defining family $(H_i,p_i)_{i\in I}$ in the sense of \thref{def::nuclearBeKo}. Since for each $i\in I$ we have that $\ker(p_i)=\{o\}$, we get $V_{p_i}=V$. This together with the fact that $V$ is dense in each $(H_i,p_i)$ ensures that the completion $(\overline{V}_{p_i},\overline{p}_i)$ is isomorphic to $(H_i,p_i)$ for all $i\in I$. Thus, for $i\leq j$ in $I$ the embedding $(H_j,p_j)\subseteq (H_i,p_i)$, which is Hilbert-Schmidt by assumption, coincides with $\overline{u}\colon (\overline{V}_{p_j},\overline{p}_j)\to(\overline{V}_{p_i},\overline{p}_i)$. Hence, $(V,\tau)$ is nuclear in the sense of \thref{def::nuclear-space2}.
\end{rem}

\begin{dfn}[c.f. \ {\cite[p.~445]{Smu17}}]
\thlabel{dfn-nuclear-Schmu}
Let $(H_n, p_n)_{i\in \NN}$ be a sequence of Hilbert spaces such that $(H_n,p_n)\subseteq(H_m,p_m)$ for all $m\leq n$ in $\NN$. The space $V:=\bigcap_{n=1}^\infty H_n$ endowed with the topology $\tau$ induced by $\P:=\{p_n:n\in\NN\}$ is called nuclear if for each $m\in\NN$ there exists $n\in\NN$ such that the embedding $(H_n,p_n)\subseteq(H_m,p_m)$ is nuclear (see \thref{def::HS+nucl}-(2)).
\end{dfn}

\begin{rem}\thlabel{schm-nucl}
Schm\"udgen's \thref{dfn-nuclear-Schmu} of nuclear space is covered by \thref{def::nuclear-space2} (and thus, by \thref{def::nuclear-space}). Indeed, let $(V,\tau)$ be a nuclear space with defining family $(H_n,p_n)_{n\in\NN}$ in the sense of \thref{dfn-nuclear-Schmu}. For each $n\in \NN$, since $\ker(p_n)=\{o\}$, we get that $V_{p_n}=V\subseteq H_n$ and so that the completion $(\overline{V}_{p_n},\overline{p}_n)$ is isomorphic to a closed subspace of $(H_n,p_n)$. As the embedding $(H_m,p_m)\subseteq (H_n,p_n)$ is nuclear, also its restriction $r$ to $\overline{V}_{p_m}$ is nuclear by \thref{restr-nuclear}. The continuity of $r$ guarantees that $r(\overline{V}_{p_m})\subseteq\overline{r(V_{p_m})}=\overline{r(V)}=\overline{V}=\overline{V}_{p_n}$ and so the map $r$ coincides with $\overline{u}\colon(\overline{V}_{p_m},\overline{p}_m)\to(\overline{V}_{p_n},\overline{p}_n)$. Hence, $\overline{u}$ is nuclear. Then \cite[Theorem 48.2]{Tr67} ensures that $\mathrm{tr}(\sqrt{\overline{u}^\ast \overline{u}})<\infty$, i.e. $\overline{u}$ is a trace-class operator (see \thref{def::trace-class}). Since the family of trace-class operators on a Hilbert space forms an ideal in the space of bounded operators on the same space (see e.g. \cite[Theorem VI.19]{Reed-Simon}), we have that $tr(\overline{u}^\ast \overline{u})<\infty$, i.e. $\overline{u}$ is Hibert-Schmidt. Thus, $(V,\tau)$ is also nuclear in the sense of \thref{def::nuclear-space2}.
\end{rem}

%If in \thref{dfn-nuclear-Schmu} we assume that $V$ is dense in $(H_n,p_n)$, then also $V_{p_n}=V$ is dense in $(H_n,p_n)$ for all $n\in\NN$. Hence, the completion $(\overline{V}_{p_n},\overline{p}_n)$ is isomorphic to $(H_n,p_n)$ and the canonical embedding $\overline{u}\colon (\overline{V}_{p_n},\overline{p}_n)\to(\overline{V}_{p_m},\overline{p}_m)$ coincides with the embedding $b_{n,m}\colon(H_n,p_n)\to(H_m,p_m)$. By \cite[Theorem 48.2]{Tr67}, the operator $b_{n,m}$ is nuclear if and only if $b_{n,m}^\ast b_{n,m}$ is nuclear. Since $b_{n,m}^\ast b_{n,m}$ is continuous, positive and self-adjoint, its nuclearity is equivalent to be a trace class operator, which in fact is equivalent to say that $b_{n,m}$ is Hibert-Schmidt. This shows that $\overline{u}$ is Hilbert-Schmidt if and only if $b_{n,m}$ is nuclear and so, in this case Schm\"udgen's \thref{dfn-nuclear-Schmu} of nuclear space is covered by \thref{def::nuclear2} (and thus, by \thref{def::nuclear-space}).
%
%
%\Maria{Are Schmuedgen assumptions on V enough to guarantee the density of $V$ in $H_n$?}

\subsection{Two auxiliary results}\label{aux}\

We provide here a proof of \eqref{eq-sigma-algebras}, which we used in the proof of \thref{MainThm-supp_2} as well as a result \thref{lem:tobbi} about dense subalgebras of topological algebras which we exploited in the analysis of \thref{cor-appl-lemma}.
\begin{proof}[Proof of \eqref{eq-sigma-algebras}]
For notational convenience, let $\tau^A:=\tau_{\sp(q)^A}$ and $\tau^B:=\tau_{\sp(q)^B}$.
We preliminarily observe that
$$q'(\alpha) := \sup_{a \in A\, : \, q(a) \leq 1} | \alpha(a) |  = \sup_{a \in B\, : \, q(a) \leq 1} | \alpha(a) |,\quad  \forall \alpha\in\sp(q)$$ and so $q'$ is lower semi-continuous w.r.t. both $ \tau^{A} $ and $ \tau^B$. Hence, all sublevel sets of $q'$ are closed in both $(\sp(q), \tau^{A})$ and $(\sp(q), \tau^{B})$, i.e. $B_n(q')^c\in \tau^A\cap \tau^B,$ for all $n\in\NN$, which gives in turn $B_n(q')\in \mathcal{B}(\tau^A)\cap \mathcal{B}(\tau^B).$ This together with the following two properties
\begin{enumerate}[label=(\roman*)]
\item $\tau^A\cap B_n(q')=\tau^B\cap B_n(q'), \quad \forall \ n\in\NN.$
\item $\mathcal{B}(\tau^C)\cap B_n(q')=\mathcal{B}(\tau^C\cap B_n(q')), \quad \forall \ n\in\NN, C\in\{A, B\}$.
\end{enumerate}
provide that
$$
\mathcal{B}(\tau^A)\cap B_n(q')\stackrel{(ii)}{=} \mathcal{B}(\tau^A\cap B_n(q'))\stackrel{(i)}{=}\mathcal{B}(\tau^B\cap B_n(q'))\stackrel{(ii)}{=}\mathcal{B}(\tau^B)\cap B_n(q')\subseteq \mathcal{B}(\tau^B).
$$
The latter ensures that if $Y\in \mathcal{B}(\tau^A)$ then $Y\cap B_n(q')\in \mathcal{B}(\tau^B)$ for all $n\in\NN$ and so that $Y=\bigcup_{n\in\NN}Y\cap B_n(q')\in \mathcal{B}(\tau^B)$, i.e. $\mathcal{B}(\tau^A)\subseteq \mathcal{B}(\tau^B)$. The opposite inclusion easily follows from $\tau^B \subset \tau^A$. Hence, $\mathcal{B}(\tau^A)= \mathcal{B}(\tau^B)$.
\end{proof}

\noindent It remains to show (i) and (ii).\\
\begin{proof}[Proof of (i)]
Let $n\in\NN$. Since $\tau^B \subset \tau^A$, we have that $\tau^A\cap B_n(q')\supseteq\tau^B\cap B_n(q').$ For the opposite inclusion, let $\alpha\in \sp(q)\cap B_n(q')$ and recall that for any $C\in\{A, B\}$ a basis of neighbourhoods of $\alpha$ in the topology $\tau^C\cap B_n(q')$ is given by $$\left\{U_{c_1, \ldots, c_k; \lambda}: k\in\NN, c_1,\ldots, c_k\in C, \lambda>0\right\},$$ where
\[
U_{c_1, \ldots, c_k; \lambda}(\alpha):= \{ \gamma \in\sp(q)\, : \, |\hat{c_j}(\gamma)-\hat{c_j}(\alpha)|< \lambda \ \mbox{for $j=1,\ldots,k$ and } q'(\gamma) < n \}
\]
We need to show that for any $k\in\NN$, $a_1,\ldots, a_k\in A$ and $\varepsilon>0$ there exist $b_1,\ldots, b_k\in B$ and $\delta>0$ such that $U_{b_1, \ldots, b_k; \delta}(\alpha)\subseteq U_{a_1, \ldots, a_k; \varepsilon}(\alpha)$.

Fixed $k\in\NN$, $a_1,\ldots, a_k\in A$ and $\varepsilon>0$, by the density of $B$ in $(A, q)$ we can always choose $b_1,\ldots, b_k\in B$ such that $q(a_j-b_j)<\frac{\varepsilon}{3n}$ for $j=1, \ldots, k$. Then taking $\delta<\frac{\varepsilon}{3}$ we have that for any $\beta\in U_{b_1, \ldots, b_k; \delta}(\alpha)$ and any $j\in\{1, \ldots, k\}$ the following holds:
\begin{eqnarray*}
|\beta(a_j)-\alpha(a_j) | &\leq& |\beta(a_j)-\beta(b_j)| + | \beta(b_j) - \alpha(b_j) | + | \alpha(b_j) - \alpha(a_j) | \\
&\leq& nq(a_j-b_j) + \delta +   nq(b_j-a_j) < \varepsilon
\end{eqnarray*}
i.e. $\beta \in U_{a_1, \ldots, a_k; \varepsilon}(\alpha)$ and hence $U_{b_1, \ldots, b_k; \delta}(\alpha)\subseteq U_{a_1, \ldots, a_k; \varepsilon}(\alpha)$.
\end{proof}

\begin{proof}[Proof of (ii)]
Let $n\in\NN$ and $C\in\{A, B\}$,

We have already showed that $B_n(q')\in\mathcal{B}(\tau^C)$ and so $\tau^C  \cap B_n(q')\subseteq\mathcal{B}(\tau^C)$, which in turn implies that $\mathcal{B}(\tau^C) \cap B_n(q') \supseteq \mathcal{B}(\tau^C \cap B_n(q'))$.

Now let $i: \sp(q)\cap B_n(q') \cap \rightarrow \sp(q)$ be the identity map. On the hand, the continuity of $i: \left(\sp(q)\cap B_n(q'), \tau^C \cap B_n(q')\right) \cap \rightarrow \left(\sp(q), \tau^C \right)$ provides that $i: \left(\sp(q)\cap B_n(q'), \mathcal{B}(\tau^C \cap B_n(q'))\right) \cap \rightarrow \left(\sp(q), \mathcal{B}(\tau^C )\right)$ is measurable. On the other hand, $\mathcal{B}(\tau^C) \cap B_n(q')$ is the smallest $\sigma-$algebra on $\sp(q)\cap B_n(q')$ making $i$ measurable.% because if this was not the case then there would exist a $\sigma-$algebra $\F\subsetneq \mathcal{B}(\tau^C) \cap B_n(q')$ making $i$ measurable, i.e. there would exist $M\in \mathcal{B}(\tau^C)$ such that $M \cap B_n(q')\notin \F$ and $i^{-1}(M)\in \F$ which is a contradiction as $i^{-1}(M)=M \cap B_n(q')$.
 Hence, we have that $\mathcal{B}(\tau^C) \cap B_n(q') \subset \mathcal{B}(\tau^C \cap B_n(q'))$.

\end{proof}

\begin{lem}\thlabel{lem:tobbi}
Let $A$ be an algebra generated by a linear subspace $V\subseteq A$ and $\tau$ a topology on $A$ such that $(A, \tau)$ is a topological algebra. If $U$ is a subspace of $V$ which is dense in $(V, \tau\restriction_V)$, then $\langle U\rangle$ is dense in $(A, \tau)$.
\end{lem}
\proof
Let $w\in U$ and $v\in V$. Then there exists a net $(v_{\alpha})_{\alpha\in I}$ with $v_{\alpha}\in U$ such that $v_{\alpha}\to v$. Since the multiplication is separately continuous, we have that $wv_{\alpha}\to wv$ and hence $wv\in\overline{\langle U\rangle}$.

Now let us take also $u\in V$.

We will show by induction on $n$ that
\begin{equation}\label{ind}
v_1, \ldots, v_n\in V \Rightarrow v_1\cdots v_n\in \overline{\langle U\rangle}, \forall n\in\NN.
\end{equation}
which implies that $\langle V\rangle =\overline{\langle U\rangle}$ and so the conclusion $A=\overline{\langle U\rangle}$.

Let us first show the base case $n=2$. If $v_1, v_2\in V$ then
Then there exist nets $(u_{\alpha})_{\alpha\in I}$ and $(w_{\beta})_{\beta\in J}$ with $u_{\alpha}, w_{\beta}\in U$ such that $u_{\alpha}\to v_1$ and $w_{\beta}\to v_2$. Since the multiplication is separately continuous, for each $u\in U$, we have that $uw_{\beta}\to uv_2$ and hence $uv_2\in\overline{\langle U\rangle}$. In particular each $u_{\alpha}v_2\in\overline{\langle U\rangle}$ and, using again the separate continuity of the multiplication, $u_{\alpha}v_2\to v_1v_2$. Hence, $v_1v_2\in\overline{\langle U\rangle}$.

Suppose now that \eqref{ind} holds for a fixed $n$ and let $v_1, \ldots, v_{n+1}\in V$. Then there exists $(g_{\alpha})_{\alpha\in I}$ with $g_{\alpha}\in U$ such that $g_{\alpha}\to v_{n+1}$. Moreover, by inductive assumption, $v_1\cdots v_n\in \overline{\langle U\rangle}$ and so there exists $(h_{\beta})_{\beta\in J}$ with $h_{\beta}\in \langle U\rangle$ such that $h_{\beta}\to v_1\cdots v_n.$ Then for any $u\in U$, by the separate continuity of the multiplication, we have that $h_{\beta}u\to v_1\cdots v_n \cdot u$ and hence $v_1\cdots v_n \cdot u\in\overline{\langle U\rangle}$. In particular for each $\alpha\in I$ we get that $v_1\cdots v_n \cdot g_\alpha\in\overline{\langle U\rangle}$. Thus, using again the separate continuity of the multiplication, we obtain that $v_1\cdots v_n \cdot g_\alpha\to v_1\cdots v_n\cdot v_{n+1}$ and so $v_1\cdots v_n\cdot v_{n+1}\in\overline{\langle U\rangle}$.
\endproof

\section*{Acknowledgments}
We are indebted to the Baden--W\"urttemberg Stiftung for the financial support to this work by the Eliteprogramme for Postdocs. This work was also partially supported by the Ausschuss f\"ur Forschungsfragen~(AFF).  Maria Infusino is member of GNAMPA group of INdAM. Tobias Kuna is member of GNFM group of INdAM.

\end{document}